\documentclass[english]{amsart}

\usepackage{enumerate}
\usepackage{amssymb}
\usepackage{array}
\usepackage{graphicx}
\usepackage{lscape}
\usepackage{mathrsfs}
\usepackage{verbatim}
\usepackage{mathtools}

\usepackage[usenames,dvipsnames,svgnames,table]{xcolor}

\usepackage{color}

\usepackage[curve,all]{xy}
\xyoption{matrix}
 \xyoption{curve}
 \xyoption{color}
 \xyoption{line}
\xyoption{arc}

\usepackage[shortlabels]{enumitem}
\setlist[enumerate]{label=\rm{(\arabic*)}}
\setlist[enumerate,2]{label=\rm({\it\roman*})}
\setlist[itemize]{label=\raisebox{0.25ex}{\tiny$\bullet$}}
\usepackage[backref=false, colorlinks, linktocpage, citecolor = blue, linkcolor = blue]{hyperref}

\theoremstyle{plain}
\newtheorem{theorem}{Theorem}[section]
\newtheorem*{theoremaux}{Theorem \theoremauxnum}
\gdef\theoremauxnum{1}

\newtheorem{proposition}[theorem]{Proposition}
\newtheorem*{propositionaux}{Proposition \propositionauxnum}
\gdef\propositionauxnum{1}

\newtheorem{lemma}[theorem]{Lemma}
\newtheorem*{lemmaaux}{Lemma \lemmaauxnum}
\gdef\lemmaauxnum{1}

\newtheorem{corollary}[theorem]{Corollary}

\theoremstyle{definition}
\newtheorem{definition}[theorem]{Definition}
\newtheorem{notation}[theorem]{Notation}
\newtheorem{example}[theorem]{Example}

\theoremstyle{remark}
\newtheorem{remark}[theorem]{Remark}

\newcommand{\leftexp}[2]{{\vphantom{#2}}^{#1}{#2}}

\newcommand{\incl}[1][r]{\ar@<-0.2pc>@{^(-}[#1] \ar@<+0.2pc>@{-}[#1]}

\newcommand{\EE}{{\sf E}}
\newcommand{\FF}{{\sf F}}
\newcommand{\GG}{{\sf G}}
\renewcommand{\AA}{{\sf A}}
\newcommand{\BB}{{\sf B}}
\newcommand{\CC}{{\sf C}}
\newcommand{\DD}{{\sf D}}
\newcommand{\Bl}{{\mathrm{Bl}}}
\newcommand{\qs}{{\dmd}}
\newcommand{\dmd}{{\text{\tiny$\diamondsuit$}}}
\renewcommand{\dmd}{{\rm qs}}
\newcommand{\phm}{\phantom{-}}
\newcommand{\phmm}{\phm\phm}
\newcommand{\pho}{\phantom{o}}
\renewcommand{\DD}{{\sf D}}

\renewcommand{\P}{\mathbb{P}}
\newcommand{\X}{\mathbb{X}}

\newcommand{\Hom}{\mathrm{Hom}}

\newcommand{\Spec}{\mathrm{Spec}}
\newcommand{\inn}{\mathrm{inn}}

\renewcommand{\SS}{\mathcal{S}}

\newcommand{\Q}{\mathbb{Q}}

\newcommand{\GGG}{\mathcal{G}}

\newcommand{\XX}{\mathcal{X}}

\renewcommand{\k}{\mathrm{k}}
\newcommand{\Dyn}{\mathrm{Dyn}}
\newcommand{\Inn}{\mathrm{Inn}}
\newcommand{\Out}{\mathrm{Out}}

\newcommand{\diag}{\mathrm{diag}}

\newcommand{\A}{\mathbb{A}}
\newcommand{\iso}{\simeq}
\newcommand{\T}{\mathbb{T}}
\newcommand{\C}{\mathbb{C}}

\newcommand{\N}{\mathbb{N}}

\newcommand{\Z}{\mathbb{Z}}
\newcommand{\R}{\mathbb{R}}
\newcommand{\G}{\mathbb{G}}
\newcommand{\Gm}{\mathbb{G}_m}

\DeclareMathOperator{\SL}{SL}
\DeclareMathOperator{\SU}{SU}
\DeclareMathOperator{\Sp}{Sp}
\DeclareMathOperator{\Aut}{Aut}

\DeclareMathOperator{\Spin}{Spin}

\DeclareMathOperator{\Gal}{Gal}
\DeclareMathOperator{\GL}{GL}

\DeclareMathOperator{\Ker}{Ker}

\newcommand{\isoto}{\xrightarrow{\sim}}

\newcommand{\hs}{\kern 0.8pt}
\newcommand{\hsss}{\kern 1.2pt}
\newcommand{\hmhh}{\kern -0.2pt}
\newcommand{\hmh}{\kern -0.4pt}
\newcommand{\hm}{\kern -0.8pt}
\newcommand{\hmm}{\kern -1.2pt}

\title[Real structures on horospherical varieties]{Real structures on horospherical varieties}

\author[Lucy Moser-Jauslin and Ronan Terpereau]{Lucy Moser-Jauslin and Ronan Terpereau
\\{\Tiny with an appendix by} Mikhail Borovoi}
\thanks{The second-named author is supported by the ANR Project FIBALGA ANR-18-CE40-0003-01.
This work received partial support from the French "Investissements d\textquoteright Avenir" program, project ISITE-BFC (contract ANR-lS-IDEX-OOOB) and from Israeli Science Foundation, grant No. 870/16.
The IMB receives support from  the EIPHI Graduate School (contract ANR-17-EURE-0002)}

\address{Institut de Math\'{e}matiques de Bourgogne, UMR 5584 CNRS, Universit\'{e} Bourgogne Franche-Comt\'{e}, F-21000 Dijon, France}
\email{lucy.moser-jauslin@u-bourgogne.fr}

\address{Institut de Math\'{e}matiques de Bourgogne, UMR 5584 CNRS, Universit\'{e} Bourgogne Franche-Comt\'{e}, F-21000 Dijon, France}
\email{ronan.terpereau@u-bourgogne.fr}



\def\ga{\,^\gamma\hskip-1pt}
\begin{document}

\begin{abstract}
We study the equivariant real structures on complex horospherical varieties, generalizing classical results known for toric varieties and flag varieties.
We obtain a necessary and sufficient condition for the existence of an equivariant real structure on a given horospherical variety, and we determine the number of equivalence classes of equivariant real structures on horospherical homogeneous spaces.
We then apply our results to classify the equivalence classes of equivariant real structures
on smooth projective horospherical varieties of Picard rank $1$.
\end{abstract}

\maketitle
\vspace{-10mm}
\tableofcontents
\vspace{-12mm}

\section*{Introduction}

A \emph{real structure} on a complex algebraic variety $X$ is an antiregular involution $\mu$ on $X$,  where \emph{antiregular} means that the following diagram commutes:
\[\xymatrix@R=4mm@C=2cm{
    X \ar[rr]^{\mu} \ar[d]  && X \ar[d] \\
    \Spec(\C)  \ar[rr]^{\Spec(z \mapsto \overline{z})} && \Spec(\C)
  }\]
Two real structures $\mu$ and $\mu'$ are called \emph{equivalent} if there exists $\varphi \in \Aut(X)$ such that $\mu'=\varphi \circ \mu \circ \varphi^{-1}$. To any real structure $\mu$ on $X$ one can associate the quotient $\XX=X/\mu$, which is a real algebraic space satisfying $\XX \times_{\Spec(\R)} \Spec(\C) \iso X$ as complex varieties. Moreover, if $X$ is quasi-projective, then $\XX$ is actually a real variety. The quotient $\XX$ is called a \emph{real form} of $X$. Two real forms $\XX$ and $\XX'$ are $\R$-isomorphic if and only if the corresponding real structures are equivalent.
Describing all the equivalence classes of real structures on a given complex variety is a classical problem in algebraic geometry.
We refer to \cite[Chp. 2]{Man20} and \cite[Chp. 3]{Ben16} for an exposition of the foundations of this theory.

When $X$ carries some extra structure, it is natural to ask for $\XX$ to also carry this extra structure. For instance, if $X=G$ is a complex algebraic group, then it is particularly interesting to describe the real structures $\sigma$ on $G$ which are group involutions, so that the real form $\GGG=G/\sigma$ is a real algebraic group (and not just a real variety). Such real structures are called \emph{real group structures}; see \S\S~\ref{sec:generalities}-\ref{sec:quasi split and inner twists} for a recap of their classification.

Another class of complex varieties carrying extra structure are the complex varieties endowed with an algebraic group action, which yields  the key notion of equivariant real structure. Let $G$ be a complex reductive algebraic group, let $\sigma$ be a real group structure on $G$, and let $X$ be a complex $G$-variety. A real structure $\mu$ on $X$ is called a $(G,\sigma)$-\emph{equivariant real structure} if $\mu(g \cdot x)=\sigma(g) \cdot \mu(x)$ for all $g \in G$ and all $x \in X$.
Two $(G,\sigma)$-equivariant real structures $\mu$ and $\mu'$ are called \emph{equivalent} if there exists $\varphi \in \Aut^G(X)$ such that $\mu'=\varphi \circ \mu \circ \varphi^{-1}$. Then the real form $\XX=X/\mu$ is a real $\GGG$-variety, and two real forms $\XX$ and $\XX'$ are isomorphic as $\GGG$-varieties if and only if the corresponding equivariant real structures are equivalent.

A \emph{horospherical subgroup} $ H \subseteq G$ is a subgroup containing a maximal unipotent subgroup of $G$, and a \emph{horospherical $G$-variety} is a normal $G$-variety with an open orbit $G$-isomorphic to $G/H$ with $H$ a horospherical subgroup of $G$. Classical examples of horospherical varieties are given by flag varieties and toric varieties. Horospherical varieties form a subclass of spherical varieties (see \cite{Pau81, Kno91}) whose combinatorial description is much more accessible.
A presentation of the theory of horospherical varieties can be found in \cite{Pas08}. Also, the combinatorial description of horospherical subgroups of $G$ from Pasquier in terms of pairs ($I,M$), where $I$ is a subset of the set of simple roots of $G$ and $M$ is a certain lattice depending on $I$, is recalled in \S~\ref{sec:setting}.

The present article aims at studying  equivariant real structures on horospherical varieties in a systematic way. Our main result is the following:
\begin{theorem} \label{th:1}
\emph{(Theorem~\ref{th:main results} and Proposition~\ref{prop:number of structures})}\
Let $G$ be a complex reductive algebraic group with a real group structure $\sigma$.
Let $H$ be a horospherical subgroup of $G$ with datum $(I,M)$.
Then a $(G,\sigma)$-equivariant real structure exists on $G/H$ if and only if $(I,M)$ is stable for the action of the Galois group $\Gal(\C/\R)$ defined by $\sigma$  and $\Delta_H(\sigma)$ is trivial, where $\Delta_H$ is the map defined by \eqref{map delta} in \S~\ref{sec: useful map}.
Moreover, if such a structure exists, then there are exactly $2^n$ equivalence classes of $(G,\sigma)$-equivariant real structures on $G/H$, where $n$ is a non-negative integer that can be calculated explicitly (see \S~\ref{sec:number of equi real structures} for details on how to compute $n$).
\end{theorem}

\begin{remark}
A couple of months after the release of the present article, Borovoi and Gagliardi obtained in \cite{BG2018} a criterion for the existence of equivariant real structures on general spherical homogeneous spaces generalizing  our Theorem~\ref{th:main results}.
\end{remark}

For brevity, we do not recall the \emph{Luna-Vust theory of spherical embeddings} (see \cite{Kno91,Tim11} for a presentation), and how to describe such embeddings in terms of the combinatorial data called \emph{colored fans}: these are fans such as those for toric varieties but with additional information called \emph{colors}.  The reader is not required to be  familiar with this theory as it is used only in \S~\ref{sec:extension of real structures}, and in that section  we describe explicitly the properties of the colored fans used.

As explained in \cite{Hur11}, if $\sigma$ is a real group structure on $G$, and if $G/H$ is a spherical homogeneous space endowed with a $(G,\sigma)$-equivariant real structure, then $\sigma$ defines an action of the Galois group $\Gal(\C/\R)$ on the set of colored fans defining a $G$-equivariant embedding of $G/H$.
The next result is an immediate consequence of \cite[Theorem 2.23]{Hur11} and \cite[Theorem 9.1]{Wed}  together with a quasi-projectivity criterion for spherical varieties due to Brion (see \S~\ref{sec:extension of real structures} for details).

\begin{theorem} \label{th:2} \emph{(Corollary~\ref{cor:extension})} \
Let $\mu$ be a $(G,\sigma)$-equivariant real structure on a horospherical homogeneous space $G/H$, and let $X$ be a horospherical $G$-variety with open orbit $G/H$. Then the real structure $\mu$ extends on $X$ if and only if the colored fan of the embedding $G/H \hookrightarrow X$ is invariant for the action of the Galois group $\Gal(\C/\R)$ defined by $\sigma$, in which case the corresponding real form $X/\mu$ is a real variety.
\end{theorem}

To illustrate the effectiveness of our results, we then consider the equivariant real structures on smooth projective horospherical $G$-varieties of Picard rank $1$ (the odd symplectic Grassmannians are examples of such varieties).

\begin{theorem} \label{th:3}
Let $G$ be a complex simply-connected semisimple algebraic group with a real group structure $\sigma$.
Let $X$ be a smooth projective horospherical $G$-variety of Picard rank $1$.
If a $(G,\sigma)$-equivariant real structure exists on $X$, then it is unique up to equivalence.
The cases where such a real structure exists are classified in Example \ref{ex:G/P 2} (when $X=G/P$ with $P$ a maximal parabolic subgroup of $G$) and in Theorem \ref{th:real forms of horo Picard 1} (when $X$ is non-homogeneous).
\end{theorem}

As mentioned before, horospherical varieties are a subclass of spherical varieties. Equivariant real structures on spherical varieties already appeared in the literature, but the scope was not the same as in this article. More precisely:
\begin{itemize}[leftmargin=*]
\item In \cite{Hur11,Wed} the authors consider the situation where a real structure on the open orbit is given and they determine in which cases this real structure extends to the whole spherical variety. They do not treat the case of equivariant real structures on homogeneous spaces.
(Note also that they work over an arbitrary field and not just over $\R$.)
\item In \cite{ACF14,Akh15,CF15} the authors study equivariant real structures on spherical homogeneous spaces $G/H$ and their equivariant embeddings when  $N_G(H)/H$ is finite. Such varieties are never horospherical, except the flag varieties.
\item In \cite{Bor20}  the author extends part of the results in \cite{ACF14,Akh15,CF15} and works over an arbitrary base field of characteristic zero.
\item In \cite{MJT} the authors obtain a criterion for the existence of equivariant real structures on symmetric spaces using the involution associated with the symmetric space instead of the homogeneous spherical data.
\end{itemize}

Besides their easy combinatorial description and their ubiquity in the world of algebraic group actions, horospherical varieties lend themselves very nicely to the study of equivariant real structures for several reasons: First of all, the group of $G$-equivariant automorphisms of $G/H$ is a torus (see \S~\ref{sec:setting}), which  reduces the computation of the number of equivalence classes of equivariant real structures on $G/H$ to the case of tori (see \S~\ref{sec:number of equi real structures}).  Secondly, if $\sigma$ is a quasi-split real group structure on $G$ such that $\sigma(H)$ is conjugate to $H$, then, by the classification of horospherical homogeneous spaces, it is easy to show that  there exists a conjugate $H'$ of $H$ such that $\sigma(H')=H'$ (Proposition \ref{prop: datum of the conjugate}); this fact, which also holds for the spherical varieties considered in \cite{ACF14} as proved there, is essential in the proof of Theorem~\ref{th:main results}. Thirdly, the Galois descent is effective for horospherical varieties, i.e. the real form $X/\mu$  is always a real variety and not only a real algebraic space as  can happen for spherical varieties (see Remark~\ref{rk:alg space}).

\smallskip

In \S\S~\ref{sec:generalities}-\ref{sec:quasi split and inner twists} we recall definitions and well-known facts about real group structures on complex algebraic groups. In \S~\ref{sec: useful map} we define the cohomologial invariant $\Delta_H(\sigma)$ that appears in the existence criterion of Theorem \ref{th:1}.

Then in \S~\ref{sec:equiv real structures} we recall the notion of equivariant real structures. In particular, we show how to determine if such a structure exists on a given homogeneous space, and if so, then how to use Galois cohomology to determine the set of equivalence classes of these equivariant real structures.

The main part of this article is \S~\ref{sec:equi real groups strcture for horo} in which we prove the results above. In \S~\ref{sec:setting} we recall the basic notions regarding the horospherical homogeneous spaces and their combinatorial classification. In \S\S~\ref{sec:quasi-split horo}-\ref{sec: non-quasi split case} we prove the existence criterion in Theorem \ref{th:1} (this is Theorem \ref{th:main results}), and in \S~\ref{sec:number of equi real structures} we prove the quantitative part in Theorem \ref{th:1} (this is Proposition \ref{prop:number of structures}). In \S~\ref{sec:extension of real structures} we recall the main result of \cite{Hur11,Wed} regarding the extension of equivariant real structures  from a spherical homogeneous space to the whole spherical variety and we apply it to prove Theorem \ref{th:2} (which is Corollary \ref{cor:extension}). Then, we apply our results to classify the equivariant real structures on smooth projective horospherical varieties of Picard rank $1$ and prove Theorem \ref{th:3} (see \S~\ref{sec: smooth proj Picard rank one}).

Finally the list of real group structures on complex simply-connected simple algebraic groups together with the list of the corresponding Tits classes is recalled in Appendix \ref{sec: tables}. These cohomology classes are useful to compute the cohomological invariant $\Delta_H(\sigma)$ in examples.

\smallskip

\noindent \textbf{Notation.}
In this article we work over the field of real numbers $\R$ and the field of complex numbers $\C$.
We denote by $\mu_n$ the group of $n$-th roots of unity and by $\Gamma$ the Galois group $\Gal(\C/\R)=\{1,\gamma\} \iso \mu_2$.
A \emph{variety} over a field $\k$ is a geometrically reduced separated scheme of finite type over $\k$; in particular, varieties can be reducible.

An \emph{algebraic group} $G$ over $\k$ is a group scheme over $\k$. By an \emph{algebraic subgroup} of $G$ we  mean a closed subgroup scheme of $G$. Reductive algebraic groups are always assumed to be connected for the Zariski topology. When we talk about a \emph{group involution} $\sigma$ we mean that $\sigma$ is an automorphism of algebraic groups (possibly over a subfield of $\k)$ such that $\sigma \circ \sigma=Id$. We refer the reader to \cite{Hum75} for the standard background on algebraic groups.

We always denote by $G$ a complex algebraic group, by $Z(G)$ its center, by $B$ a Borel subgroup of $G$, by $T$ a maximal torus of $B$, and by $U$ the unipotent radical of $B$ (which is also a maximal unipotent subgroup of $G$). If $H$ is a subgroup of $G$, then $N_G(H)$ denote the  normalizer of $H$ in $G$. We write $\X=\X(T)=\Hom_{gr}(T,\Gm)$ for the character group of $T$ and $\X^\vee=\X^\vee(T)=\Hom_{gr}(\Gm,T)$ for the cocharacter group of $T$. When $G$ is semisimple we denote by $\Dyn(G)$ its Dynkin diagram.

\section{Real group structures}

In this section, we start by recalling definitions and well-known facts about real group structures on complex algebraic groups. We are mostly interested in the case of reductive groups. In particular, we show how to obtain all real group structures on complex reductive algebraic groups, by piecing together the structures on complex tori and on complex simply-connected simple algebraic groups.

The main references we use are \cite{Con14} for results concerning the structure of reductive algebraic groups, \cite{Ser02} for general results concerning Galois cohomology and  \cite{Man20,Ben16} for generalities on real structures. In \textit{loc.cit.} the author  treats the case of real structures on complex quasi-projective varieties; the  corresponding results referred to here concern complex algebraic groups  and can be treated in exactly the same way.

\subsection{Generalities and first classification results} \label{sec:generalities}

\begin{definition} (Real group structures on complex algebraic groups.)
\begin{enumerate}[(i),leftmargin=*]
\item Let $G$ be a complex algebraic group. A \emph{real group structure on $G$} is an antiregular group involution $\sigma: G \to G$, i.e., a group involution over $\Spec(\R)$ which makes the following diagram commute:
\[\xymatrix@R=4mm@C=2cm{
    G \ar[rr]^{\sigma} \ar[d]  && G \ar[d] \\
    \Spec(\C) \ar[rd] \ar[rr]^{\Spec(z \mapsto \overline{z})} && \Spec(\C)  \ar[ld]\\
       & \Spec(\R)
  }\]
\item Two real group structures $\sigma$ and $\sigma'$ on $G$ are \emph{equivalent} if there exists a (regular) group automorphism $\varphi \in \Aut_{gr}(G)$ such that $\sigma'=\varphi \circ \sigma \circ \varphi^{-1}$.
\end{enumerate}
\end{definition}

\begin{remark} \label{rk:any antiregular group automorphism is of the form}
If $\sigma$ is an antiregular group involution on $G$, then any antiregular group automorphism is of the form $\varphi\circ\sigma$ for some group automorphism $\varphi$. If $\varphi$ and $\sigma$ commute, then $\varphi\circ\sigma$ is an involution if and only if $\varphi$ is an involution.
\end{remark}

If $(G,\sigma)$ is a complex algebraic group with a real group structure, then the quotient scheme $\GGG=G/\sigma$ is a real algebraic group which satisfies $\GGG \times_{\Spec(\R)} \Spec(\C) \iso G$ as complex algebraic groups.  The real group $\GGG$ is called a \emph{real form of $G$}. Two real forms are $\R$-isomorphic if and only if the corresponding real group structures are equivalent (see \cite[Corollary 3.13]{Ben16}).

\begin{definition} \label{def:real part}
If $(G,\sigma)$ is a complex algebraic group with a real group structure, then $G_0=G(\C)^\sigma$ is called the \emph{real part} (or \emph{real locus}) of $(G,\sigma)$; it coincides with the set of $\R$-points of the real algebraic group $G/\sigma$ (see \cite[Proposition 3.14]{Ben16} for details).
\end{definition}

With the notation of Definition \ref{def:real part} the group of $\C$-points $G(\C)$  is a complex Lie group and $G_0=G(\C)^\sigma$ is a real Lie group.
In fact, the real part $G_0$ determines the real group structure for a connected complex algebraic group:

\begin{theorem} \label{th: real part determines real structure} \emph{(\cite[Theorem 3.41]{Ben16})}
Let  $G$ be a connected complex algebraic group with two real group structures $\sigma$ and $\sigma'$.
Then $\sigma$ and $\sigma'$ are equivalent if and only if there is a (scheme) automorphism $\varphi  \in \Aut(G)$ such that $\varphi(G(\C)^{\sigma})=G(\C)^{\sigma'}$.
\end{theorem}

Because of Theorem \ref{th: real part determines real structure} we will sometimes consider the real part $G_0=G(\C)^\sigma$ instead of the real group structure $\sigma$ when describing all the possible equivalence classes of real group structures on connected complex algebraic groups.

\bigskip

Let $G$ be a complex reductive algebraic group, let $T=Z(G)^0$ be the neutral component of the center of $G$, and let $G'$ be the derived subgroup of $G$. Then the homomorphism $T \times G' \to G$, $(t,g') \mapsto t^{-1}g'$ is a central isogeny with kernel $T \cap G'$. Also, there is a 1-to-1 correspondence
\[ \left \{ \text{real group structures $\sigma$ on $G$} \right \}  \leftrightarrow \left \{ \begin{tabular}{l}
\text{real group structures $(\sigma_1,\sigma_2)$ on $T \times G'$} \\
\text{ \ \ such that $\sigma_{1| T \cap G'}= \sigma_{2| T \cap G'}$}
\end{tabular} \right \} \]
given by $\sigma \mapsto (\sigma_{|T},\sigma_{|G'})$. Therefore, to determine real group structures on complex reductive algebraic groups, it suffices to determine real group structures on complex tori and on complex semisimple algebraic groups.

\begin{lemma} \label{lem: real form on tori} \emph{(Classification of real group structures on complex tori.)}\\
Let $T \iso \G_m^n$ be an $n$-dimensional complex torus.
\begin{enumerate}[(i),leftmargin=*]
\item If $n=1$, then $T$ has exactly two inequivalent real group structures, defined by $\sigma_0: t \mapsto \overline{t}$ and $\sigma_1: t \mapsto \overline{t}^{-1}$.
\item If $n=2$, then $\sigma_2:  (t_1,t_2) \mapsto (\overline{t_2},\overline{t_1})$ defines a real group structure on $T$.
\item \label{item: n at least 2} If $n\ge 2$,  then every real group structure on $T$ is equivalent to exactly one real group structure of the form  $\sigma_0^{\times n_0}\times\sigma_1^{\times n_1}\times\sigma_2^{\times n_2}$, where $n=n_0+n_2+2n_2$.
\end{enumerate}
\end{lemma}

\begin{proof}
This result is well-known to specialists (see for instance \cite[Chp.4, \S~10.1]{Vos98}) but we give a sketch of the proof for the sake of completeness.

Clearly, each $\sigma_i$ for $i=1,2$ or $3$ defines a real group structure on $T$. For $n=1$, the structures $\sigma_0$ and $\sigma_1$ are inequivalent, since the real parts are not diffeomorphic. Also since $\Aut_{gr}(\G_m) \iso \mu_2$, these are the only two real group structures on a one-dimensional torus. For dimension two,  $\sigma_2$ defines a new real group structure since it is an antiregular group involution.
Also, for $n\ge2$  all the real group structures on $T$ given in \ref{item: n at least 2} exist and are inequivalent since the corresponding real parts are $(\R^*)^{n_0} \times (S^1)^{n_1} \times (\C^*)^{n_2}$ which are pairwise non-diffeomorphic.

It remains to show that all real group structures are equivalent to one of the structures given in (iii). This is done by determining all the conjugacy classes of $\Aut_{gr}(\G_m^n)\iso \GL_n(\Z)$.
More precisely, any real group structure in dimension $n$ is equivalent to $\sigma=\varphi\circ (\sigma_0^{\times n})$ for some $\varphi\in \Aut_{gr}(\G_m^n)$ (see Remark \ref{rk:any antiregular group automorphism is of the form}). Since all group
automorphisms commute with $\sigma_0^{\times n}$, we see that  $\sigma$ is an involution if and only if $\varphi$ is an involution. Also, two involutions define equivalent real group structures if and only if they are in the same conjugacy class.

Finally, note that the conjugacy classes  of elements of order $2$ in $\GL_n(\Z)$ are represented exactly by block diagonal matrices of the form
\[ \diag \left(   1,\ldots,1,-1,\ldots,-1, \begin{bmatrix}
0 & 1 \\1 & 0
\end{bmatrix},\ldots, \begin{bmatrix}
0 & 1 \\1 & 0
\end{bmatrix} \right).\]
Each block corresponds to $\sigma_0$, $\sigma_1$ and $\sigma_2$ respectively, which proves the result.
\end{proof}

\begin{remark}\label{rem:torus-variety}
If we forget the group structure and considers $T \iso \G_m^n$ as a complex variety, then there are other real structures on $T$. For instance $\tau_1: t \mapsto -\overline{t}^{-1}$ is a real structure on $\G_m$ but not a real group structure. In fact, one can show that any real structure on $T$ is equivalent to a product $\sigma_0^{\times n_0} \times \sigma_1^{\times n_1} \times \sigma_2^{\times n_2} \times \tau_1^{\times m}$ for some $n_0, n_1,n_2, m \in \N$.
\end{remark}

It remains to determine the real group structures on complex semisimple algebraic groups. For any complex semisimple algebraic group $G'$, there exists a central isogeny $\varphi: \widetilde{G'} \to G'$, where $\widetilde{G'}$ is a simply-connected semisimple algebraic group (uniquely defined by $G'$ up to isomorphism); see \cite[Exercise 1.6.13]{Con14}.
Then $\widetilde{G'}$ is isomorphic to a product of simply-connected simple algebraic groups \cite[\S~6.4]{Con14}; the later is uniquely defined up to permutation of the factors.

The next lemma is also well-known but for sake of completeness we recall the proof. It reduces the classification of real group structures on simply-connected semisimple groups to the classification of real group structures on simply-connected simple groups.

\begin{lemma} \label{lem:easy_lemma_reduction}
Let $\sigma$ be a real group structure on a complex simply-connected semisimple algebraic group $G' \iso \prod_{i \in I} G_i$, where the $G_i$ are the simple factors of $G'$. Then, for a given $i \in I$, we have the following possibilities:
\begin{enumerate}[(i),leftmargin=*]
\item \label{item: Gi stable}$\sigma(G_i)=G_i$ and $\sigma_{|G_i}$ is a real group structure on $G_i$; or
\item \label{item:Gi and Gj swap}  there exists $j \neq i$ such that $\sigma(G_i)=G_j$, then $G_i \iso G_j$ and $\sigma_{| G_i \times G_j}$ is equivalent to $(g_1,g_2) \mapsto (\sigma_0(g_2),\sigma_0(g_1))$, where $\sigma_0$ is an arbitrary real group structure on $G_i \iso G_j$.
\end{enumerate}
\end{lemma}

\begin{proof}
We use the fact that the factors $G_i$ are the unique simple normal subgroups of $G$ (see \cite[Theorem 5.1.19] {Con14}). In particular any group automorphism of $G$ permutes the factors. Since $\sigma$ is a group involution,
 either $\sigma(G_i)=G_i$ and we get \ref{item: Gi stable}, or $\sigma(G_i)=G_j$ for some $j \neq i$. In the second case,  $G_i$
 and $G_j$ are $\R$-isomorphic, and since they are both simply-connected simple groups they must be $\C$-isomorphic (this
 follows for instance from the classification  of simply-connected simple algebraic groups in terms of Dynkin diagrams). Therefore
 $G_i \times G_j \iso H \times H$, for some simply-connected simple group $H$, and $\sigma_{|G_i \times G_j}$ identifies with $
 \sigma_{H \times H}: (h_1,h_2) \mapsto (\sigma_1(h_2),\sigma_1^{-1}(h_1))$ for some antiregular automorphism $\sigma_1$ on $H$. But
 then it suffices to conjugate $\sigma_{H \times H}$ with the group automorphism $\varphi$ defined by  $(h_1,h_2) \mapsto (\sigma_1 \circ
 \sigma_0(h_2),h_1)$ to get the real group structure $(g_1,g_2) \mapsto (\sigma_0(g_2),\sigma_0(g_1))$, where $\sigma_0$ is an arbitrary real group structure on $H$.
\end{proof}

The real group structures on complex simply-connected simple algebraic groups are well-known (a recap can be found in Appendix \ref{sec: tables}); they correspond to real Lie algebra structures on complex simple Lie algebras (see \cite[\S~V\!I.10]{Kna02} for the classification of those).
Therefore we can determine all the real group structures on complex simply-connected semisimple algebraic groups from Lemma \ref{lem:easy_lemma_reduction}.  In the next subsection, we will give a brief outline of a way to classify them, using quasi-split structures and inner twists.

\begin{example}
Up to equivalence, there are two real group structures on $\SL_2$ given by $\sigma_0(g)=\overline{g}$ and $\sigma_1(g)=\leftexp{t}{\overline{g}}^{-1}$.
Up to equivalence, there are four real group structures on $\SL_2 \times \SL_2$ given by $\sigma_i \times \sigma_j$ with $(i,j) \in \{(0,0),(0,1),(1,1)\}$ and $\sigma_2: (g_1,g_2) \mapsto (\sigma_0(g_2),\sigma_0(g_1))$.
Similarly, we let the reader check that, up to equivalence, there are six real group structures on $\SL_2 \times \SL_2 \times \SL_2$ and nine real group structures on $\SL_2 \times \SL_2 \times \SL_2 \times \SL_2$.
\end{example}

\subsection{Quasi-split real group structures and inner twists} \label{sec:quasi split and inner twists}
Let $(G,\sigma)$ be a complex reductive algebraic group with a real group structure. Note that the set of Borel subgroups of $G$ and, for any $n \in \N$, the set of
$n$-dimensional
tori of $G$ are each preserved by $\sigma$ (this is a direct consequence of the definition of Borel subgroups and tori).

\begin{definition}\textcolor{white}{--}
\begin{enumerate}[(i),leftmargin=*]
\item If there exists a Borel subgroup $B \subseteq G$ such that $\sigma(B)=B$, then $\sigma$ is called \emph{quasi-split}.
Let $T \subseteq B$ be a maximal torus such that $\sigma(T)=T$ (such a torus always exists by Theorem \ref{th:ABC Galois theory} \ref{item: one qs in each fiber}). With the notation of Lemma \ref{lem: real form on tori}, if the restriction $\sigma_{|T}$ is equivalent to a product $\sigma_0^{\times \dim(T)}$, then $\sigma$ is called \emph{split}.
\item For $c \in G$ we denote by $\inn_c$ the inner automorphism of $G$ defined by
\[ \inn_c: G \to G, g \mapsto cgc^{-1}.\]
If $\sigma_1$ and $\sigma_2$ are two real group structures on $G$ such that $\sigma_2 = \inn_c \circ \sigma_1$, for some $c \in G$, then $\sigma_2$ is called an \emph{inner twist} of $\sigma_1$.
\end{enumerate}
\end{definition}

\begin{remark} \label{rk:trivial split action}
If $\sigma$ is a quasi-split real group structure on $G$, then for any pair $(B',T')$ with $B' \subseteq G$ a Borel subgroup and $T' \subseteq B'$ a maximal torus, there exists an equivalent real group structure $\sigma'=\inn_c \circ \sigma \circ \inn_c^{-1}$, for some $c \in G$, stabilizing both $B'$ and $T'$.
 \end{remark}

There exists a unique split real group structure on $G$ up to equivalence (see \cite[Theorem 6.1.17]{Con14} or \cite[Chp.5, \S~4.4]{OV90} when $G$ is semisimple) that we will always denote by $\sigma_0$. There is also a unique compact real group structure $\sigma_c$ on $G$ up to equivalence (see \cite[ Chp. 5, \S\S~1.3-1.4]{OV90} when $G$ is semisimple and Lemma \ref{lem: real form on tori} when $G$ is a torus, the general case follows from these two cases).

In general, $G$ may have several inequivalent quasi-split real group structures. If $G$ is a simple algebraic group, then $G$ has at most two inequivalent quasi-split real group structures (this will follow from Theorem \ref{th:ABC Galois theory} \ref{item: dynkin conjugacy classes}). On the other hand, for tori all real group structures are quasi-split.

The next classical lemma yields a description of the set of real group structures obtained as inner twists of a given real group structure.

\begin{lemma} \label{lem: parametrization of the inner twists}
For a given $c \in G$, the antiregular group automorphism $\inn_c \circ \sigma$ is a real group structure on $G$ if and only if $c \sigma(c) \in Z(G)$ (and then $c \sigma(c)= \sigma(c)c)$. Also, $\inn_c \circ \sigma=\inn_{c'} \circ \sigma$ if and only if $c^{-1}c' \in Z(G)$.
\end{lemma}

\begin{proof}
Since $\inn_c \circ \sigma$ is an antiregular group automorphism, it is a real group structure if and only if it is an involution, i.e.
\begin{align*}
(\inn_c \circ \sigma) \circ (\inn_c \circ \sigma)=Id & \Leftrightarrow \forall g \in G, c \sigma(c \sigma(g) c^{-1}) c ^{-1}=g \\
                                & \Leftrightarrow \forall g \in G, c \sigma(c) g =g c \sigma(c) \\
                                                                & \Leftrightarrow c \sigma(c) \in Z(G).
\end{align*}
A similar computation yields the second equivalence stated in the lemma.
\end{proof}

Note that in the previous lemma, we do not give a general condition describing when $\inn_c\circ\sigma$ is equivalent to $\inn_{c'}\circ\sigma$.

\begin{example} \label{ex: forms of SL3}
Let $c=c^{-1}=\begin{bmatrix}
0 & 0 & -i\\ 0 & -1 & 0 \\ i & 0 & 0
\end{bmatrix}\in \SL_3$.
Up to equivalence, the group $\SL_3$ has three real group structures given by $\sigma_0(g)=\overline{g}$, which is split and whose real part is $\SL_3(\R)$,
$\sigma_1(g)=c(\leftexp{t}{\overline{g}}^{-1})c^{-1}$, which is quasi-split and whose real part is $\SU(1,2)$, and  $\sigma_2(g)=\leftexp{t}{\overline{g}}^{-1}=\inn_c \circ \sigma_1$, which is an inner twist of $\sigma_1$ and whose real part is $\SU(3)$.
\end{example}

We now recall a $\Gamma$-action that will play an important role in \S~\ref{sec:equi real groups strcture for horo} when studying the equivariant real structures on horospherical varieties.

\begin{definition} \label{def: Gamma-action on X(T)}
Let $(G,\sigma_{qs})$ be a complex reductive algebraic group with a quasi-split real group structure that preserves a Borel subgroup $B \subseteq G$ and a maximal torus $T \subseteq B$. It induces a $\Gamma$-action on $\X$ and on $\X^\vee$ as follows:
 \[\forall \chi \in \X,\ \ga \chi=\tau \circ \chi \circ \sigma_{qs} \ \ \ \ \text{ and } \ \ \ \ \forall \lambda \in \X^\vee,\ \ga \lambda=\sigma_{qs} \circ \lambda \circ \tau\, ,\]
 where $\tau(t)=\overline{t}$ is  complex conjugation.
\end{definition}

Let us note that the sets of roots, coroots, simple roots, and simple coroots associated with the triple $(G,B,T)$ are preserved by this $\Gamma$-action \cite[Remark 7.1.2]{Con14}, and so $\Gamma$ acts on the based root datum of $G$. Moreover, if $\sigma_{qs}=\sigma_0$ is a split real group structure, then the $\Gamma$-action on $\X$ and $\X^\vee$ is trivial.

\smallskip

We now recall the definition of the first Galois cohomology pointed set as it will appear several times in the rest of this article. Since we are concerned with real structures, we will restrict the presentation to Galois cohomology for $\Gamma$-groups. More details on Galois cohomology in a more general setting can be found in \cite{Ser02}.

\begin{definition} \label{def:Galois H1}
If $A$ is a $\Gamma$-group, then  the first Galois cohomology pointed set is $H^1(\Gamma,A)=Z^1(\Gamma,A)/\sim$, where $Z^1(\Gamma,A)=\{ a \in A \ | \   a^{-1}= \ga a \}$ and two elements $a_1$, $a _2 \in Z^1(\Gamma,A)$ satisfy $a_1 \sim a_2$ if $a_2=b^{-1} a_1\ga  b$ for some $b \in A$.
\end{definition}

\begin{remark}\label{rem:cohom-ab}
If $A$ is an abelian group, then $H^1(\Gamma, A)$ is an abelian group.
In this case, we can also define $H^2(\Gamma,A)$ which identifies with the group  $A^\Gamma/\{a \ga a\ |\ a\in A\}$; see \cite[\S~I.2]{Ser02} for details.
\end{remark}

\begin{remark} \label{rk:2 torsion}
We have $a^2 =a (a^{-1})^{-1}=a(\ga a)^{-1} \sim 1$ for all $a \in Z^1(\Gamma,A)$. In the case where $H^1(\Gamma, A)$ is finite, this implies that its cardinal is a power of $2$.
\end{remark}

\begin{notation}
The Galois cohomology obviously depends on the $\Gamma$-action on the group $A$. In this article, we will sometimes consider different $\Gamma$-actions for the same group $A$. If the action is not clear from the context, then we will specify the action by  writing: $H^i(\Gamma,A)$ for the $\Gamma$-action on $A$ induced by $\sigma$; this means that $\sigma$ is an involution on $A$, and that the non-trivial element $\gamma\in\Gamma$ acts on $A$ by applying $\sigma$.

For example, if $A$ an automorphism group of a variety $X$, and $\mu$ is an involution on $X$, then when we write $H^i(\Gamma,A)$ for the $\Gamma$-action on $A$ induced by $\mu$-conjugation, we mean that the non-trivial element $\gamma\in\Gamma$ acts on $A$ by conjugating automorphisms by $\mu$.
\end{notation}

We now calculate the first and second cohomology groups for the case where $A=\T$ is a torus, and the $\Gamma$-action on $\T$ is induced by a real group structure.
The result for the first cohomology group will be used later in the article, in Lemma \ref{lem:injective-chi} and in Proposition \ref{prop:number of structures}. We will use the result for the second cohomology group in Remark \ref{rem:n0}.

\begin{proposition}\label{prop:torus-cohom}
Let $\T$ be a torus endowed with a real group structure equivalent to a product $\sigma_0^{\times n_0} \times \sigma_1^{\times n_1} \times \sigma_2^{\times n_2}$ (with the notation of Lemma \ref{lem: real form on tori}), then
\begin{itemize}
\item[(i)] $H^1(\Gamma, \T) \iso (\mu_2)^{n_1}$; and
\item[(ii)] $H^2(\Gamma, \T) \iso (\mu_2)^{n_0}$.
\end{itemize}

\end{proposition}

\begin{proof}  This result comes from a direct computation. More precisely, if   $\T \iso \G_m$ with the real structure $\sigma_0$, given $a\in \T^\Gamma$, there exists $b\in\T$ such that $a=\pm b\ga b$. If $\T \iso \G_m$ with the real structure $\sigma_1$, any $a\in\T^\Gamma$ is of the form $b\ga b$ for a choice of $b\in T$, and the same holds for $\T\iso\G_m^2$ endowed with the real structure $\sigma_2$.  Now by using the definitions of the two cohomology groups, one finds that each $\sigma_0$-component of the real structure induces a non-trivial component of $H^2(\Gamma,\T)$, isomorphic to $\mu_2$, and each $\sigma_1$-component of the real structure induces a non-trivial component of  $H^1(\Gamma,\T)$, isomorphic to $\mu_2$. The $\sigma_0$- and $\sigma_2$-components have no effect on the first cohomology group, and the  $\sigma_1$- and $\sigma_2$-components have no effect on the second cohomology group.
\end{proof}

Let $\sigma$ be a real group structure on the complex reductive algebraic group $G$.
As $\sigma$ is a group involution, it preserves $Z(G)$.
Let $\Inn(G) \iso G/Z(G)$ be the group of inner automorphisms of $G$ and let $\Out(G)$ be the quotient group $\Aut_{gr}(G)/\Inn(G)$.
The Galois group $\Gamma$ acts on $\Aut_{gr}(G)$ by $\sigma$-conjugation, i.e. $\ga \varphi=\sigma \circ \varphi \circ \sigma$ for all $\varphi \in \Aut_{gr}(G)$; this $\Gamma$-action stabilizes $\Inn(G)$ on which it coincides with the $\Gamma$-action induced by $\sigma$ on $G/Z(G)$.
The short exact sequence
\begin{equation*}
1 \to \Inn(G) \to \Aut_{gr}(G) \to \Out(G)\to 1
\end{equation*}
induces a long exact sequence in Galois cohomology (see \cite[\S~5.5]{Ser02}):
\[ \hspace{-9mm} 1 \to \Inn(G)^\Gamma \to \Aut_{gr}(G)^\Gamma \to \Out(G)^\Gamma  \to H^1(\Gamma,\Inn(G)) \to H^1(\Gamma,\Aut_{gr}(G)) \xrightarrow[]{\kappa} H^1(\Gamma,\Out(G)).\]

The next theorem gathers the main results we will need regarding the classification of the real group structures on $G$ via Galois cohomology.

\begin{theorem} \label{th:ABC Galois theory}
We keep the previous notation, and we fix a Borel subgroup $B \subseteq G$ and a maximal torus $T \subseteq B$. The following statements hold:
\begin{enumerate}[(i),leftmargin=*]
\item  \label{item:correspondance H1 and real structures} Let $\Gamma$ act on $ \Aut_{gr}(G)$  by $\sigma$-conjugation as above. Then the map
\begin{equation*}
 H^1(\Gamma,\Aut_{gr}(G)) \to \{\text{real group structures on $G$}\}/\text{equiv} \ \ \text{induced by}\  \varphi \mapsto \varphi \circ \sigma
 \end{equation*}
is a bijection that sends the identity element to the equivalence class of $\sigma$.
\item  \label{item: fibers parametrized inner twists} We have $\kappa(\sigma_1)=\kappa(\sigma_2)$ if and only if $\sigma_2$ is equivalent to an inner twist of $\sigma_1$.
\item \label{item: one qs in each fiber}There is exactly one equivalence class of quasi-split real group structures in each non-empty fiber of the map $\kappa$. Moreover, in each of them there is a quasi-split real group structure that stabilizes $B$ and $T$.
\item \label{item: dynkin conjugacy classes} If moreover $G$ is simply-connected semisimple, then the map
\[ \{ \text{quasi-split real group structures on $G$ preserving $B$ and $T$}\} \to \Aut(\Dyn(G)) \]
of Definition \ref{def: Gamma-action on X(T)} induces a bijection between the set of equivalence classes of quasi-split real group structures and the set of conjugacy classes of elements of order $\leq 2$ in $\Aut(\Dyn(G))$.
\end{enumerate}
\end{theorem}

\begin{proof}
All the proofs and details of the statements can be found in \cite[\S~7]{Con14}.
\end{proof}

\begin{example} \item
\begin{enumerate}[leftmargin=*]
\item We keep the notation of Example \ref{ex: forms of SL3} and we write $\sigma_{ij}=\sigma_i \times \sigma_j$.
The group $G=\SL_3 \times \SL_3$ has seven inequivalent real group structures:  $\sigma_{00}$ (split), $\sigma_{01}$ (quasi-split) and its inner twist $\sigma_{02}=\inn_{(1,c)} \circ  \sigma_{01}$, $\sigma_{11}$ (quasi-split) and its two inner twists $\sigma_{12}=\inn_{(1,c)} \circ \sigma_{11}$ and $\sigma_{22}=\inn_{(c,c)} \circ \sigma_{11}$, and the quasi-split real group structure $\sigma:(g_1,g_2) \mapsto (\sigma_0(g_2),\sigma_0(g_1))$. Also, $\Aut(\Dyn(G))$ is the dihedral group $\left \langle r,s \ | \ r^4=s^2=(sr)^2=1 \right  \rangle$ which is the union of five conjugacy classes $\{1\}$, $\{r^2\}$, $\{s,rsr^{-1}\}$, $\{ sr, rs  \}$, and $\{r,r^{-1}\}$, four of which consist of  elements of order $\leq 2$. We may check that the bijection in  Theorem \ref{th:ABC Galois theory} \ref{item: dynkin conjugacy classes} is the following one:
$\sigma_{00} \leftrightarrow \{1\}$,
$\sigma_{01}  \leftrightarrow \{sr,rs\}$,
$\sigma_{11}  \leftrightarrow \{r^2\}$,
and $\sigma  \leftrightarrow \{s,rsr^{-1}\}$.

\smallskip

\item Let $G=\Spin_8$.  Then $G$ has a unique (up to equivalence) split real group structure denoted, as usual, by $\sigma_0$. Note that $\Dyn(G)=D_4$ and that $\Aut(\Dyn(G)) \iso \mathfrak{S}_3$.
This group has exactly one conjugacy class of non-trivial elements of order $2$, which corresponds via the bijection in  Theorem \ref{th:ABC Galois theory} \ref{item: dynkin conjugacy classes} to the unique (up to equivalence) quasi-split non-split real group structure $\sigma_1$. Moreover, there are three other inequivalent real group structures on $G$, which are both inner twists of $\sigma_0$, and one other inequivalent real group structure on $G$ which is an inner twist of $\sigma_1$.  The real parts of these real groups structures are  the real Lie groups $\Spin^*(8)$ and  $\Spin(8-m,m)$ with $0\le m\le 4$.  The case $m=4$ corresponds to $\sigma_0$, and its inner twists correspond to $m=2$ and $m=0$ (the later is the real part of the compact real group structure on $G)$ and $\Spin^*(8)$. The case $m=3$ corresponds to $\sigma_1$, and  the case $m=1$ corresponds to the inner twist of $\sigma_1$.
\end{enumerate}
\end{example}

\subsection{A cohomological invariant} \label{sec: useful map}
Let $(G,\sigma_{qs})$ be a complex reductive algebraic group with a quasi-split real group structure.
We consider the short exact sequence of $\Gamma$-groups
\[1\to Z(G)\to G\to G/Z(G)\to 1,\]
where the $\Gamma$-action is induced by $\sigma_{qs}$. More precisely, the element $\gamma\in\Gamma$ acts on $G$ and $Z(G)$ by $\sigma_{qs}$, and on $G/Z(G)$ by the induced real group structure  $\overline{\sigma_{qs}}$.
Since $Z(G)$ is an abelian group, there is a connecting map (see \cite[\S~I.5.7]{Ser02})
\begin{equation*}
\delta: H^1(\Gamma, G/Z(G))\to H^2(\Gamma,Z(G)).
\end{equation*}
It follows from Lemma \ref{lem: parametrization of the inner twists} and Definition \ref{def:Galois H1} that $Z^1(\Gamma, G/Z(G))$ identifies with the set of inner twists of $\sigma_{qs}$ and
\[H^1(\Gamma, G/Z(G)) \iso \{ c \in G \ | \ c \sigma_{qs}(c) \in Z(G) \}/\sim \]
where $c \sim c'$  if $c^{-1}b^{-1}c' \sigma_{qs}(b) \in Z(G)$ for some $b \in G$.
Also, there is an isomorphism of abelian groups (see Remark \ref{rem:cohom-ab}):
\[ H^2(\Gamma,Z(G)) \iso Z(G)^\Gamma/\{a \sigma_{qs}(a) \ | \  a \in Z(G)\}.\]
With these identifications, the connecting map $\delta$ is the map induced by
\[  \{ c \in G \ | \ c \sigma_{qs}(c) \in Z(G) \} \to Z(G)^\Gamma,\; c \mapsto c \sigma_{qs}(c).\]
If $\sigma$ is a real group structure on $G$ equivalent to $\inn_c \circ \sigma_{qs}$, then we will also write  $\delta(\sigma)$ instead of $\delta(\overline{c})$.

\begin{definition}
When $G$ is a simply-connected semisimple algebraic group, the element $\delta(\sigma)$ is called the \emph{Tits class} of the real group structure $(G,\sigma)$.
\end{definition}

Tables where the Tits classes are determined  for any $(G,\sigma)$, with $G$ a simply-connected simple algebraic group and $\sigma$ a real group structure on $G$, can be found in Appendix \ref{sec: tables}.

\smallskip

Consider now the case where $H$ is a subgroup of $G$ such that:
\begin{itemize}
\item the algebraic group $N_G(H)/H$ is abelian (this is the case, for example, if $H$ is a spherical subgroup of $G$ by \cite[Proposition 3.4.1]{Per14}); and
\item $H$ is conjugate to a subgroup $H'$ which is stable by $\sigma_{qs}$.
\end{itemize}
Replacing $H$ by $H'$, we may assume that $\sigma_{qs}(H)=H$ to simplify the situation.
Then  $\sigma_{qs}$ induces a real group structure on $N_G(H)/H$, namely $\overline{\sigma_{qs}}(nH)=\sigma_{qs}(n)H$, and we can consider the second cohomology group $H^2(\Gamma, N_G(H)/H)$. The natural homomorphism $\chi_H:Z(G)\to N_G(H)/H$, induced by the inclusion $Z(G) \to N_G(H)$,
yields an homomorphism between the second cohomology groups
\[ \chi_H^*:H^2(\Gamma, Z(G)) \to H^2(\Gamma, N_G(H)/H).\]
In the rest of this article we will denote the composed map $\chi_H^* \circ \delta$ by
\begin{equation} \label{map delta} \tag{\textasteriskcentered}
\Delta_H: H^1(\Gamma, G/Z(G))  \to H^2(\Gamma, N_G(H)/H).
 \end{equation}
The element $\Delta_H(\sigma)  \in H^2(\Gamma, N_G(H)/H)$ is the \emph{cohomological invariant} referred to in the title of this subsection.

\begin{remark}\label{rem:n0}
A consequence of Proposition \ref{prop:torus-cohom} is that if $N_G(H)/H$ is a torus, and if the $\Gamma$-action is induced by a real group structure on this torus with $n_0=0$, then $H^2(\Gamma, N_G(H)/H)$ is trivial and so $\Delta_H$ is the trivial map.
\end{remark}

\begin{lemma} \label{lem:injective-chi}
With the previous notation, let $H$ be a maximal unipotent subgroup of $G$. Then the two conditions above are satisfied and $\Delta_H(\sigma)$ is trivial if and only if $\delta(\sigma)$ is trivial.
\end{lemma}

\begin{proof}
First, note that $\sigma_{qs}$ stabilizes a Borel subgroup $B$, and therefore also its maximal unipotent subgroup $U$. Also, $N_G(H)=B$, and therefore $N_G(H)/H=\T$ is a torus (isomorphic to a maximal torus of $G)$. Thus the two conditions above hold.
By applying an appropriate conjugation, we can assume that $H=U$ and then $\sigma_{qs}(H)=H$.

Now we will show that $\chi_U^*$ is injective, which will imply the result.
Consider the short exact sequence
\vspace{-2mm}
\[0\mapsto Z(G)\to \T\to\overline{\T}=\T/Z(G)\to 0.\]
This exact sequence induces an exact sequence of cohomology groups:
\[H^1(\Gamma,\overline{\T})\to H^2(\Gamma,Z(G))\to H^2(\Gamma,\T),\]
where the second map is simply $\chi_U^*$.
The torus $\overline{\T}$ is isomorphic to a maximal torus of the adjoint semi-simple group $G/Z(G)$ and the quasi-split real group structure $\overline{\sigma_{qs}}$ on $G/Z(G)$, induced by $\sigma_{qs}$ on $G$, acts on the character group of $\overline{\T}$ by permutations. Indeed, since $\overline{\sigma_{qs}}$ stabilizes the Borel subgroup $\overline{B}$ of $G/Z(G)$, it preserves the positive roots of $(G/Z(G),\overline{B},\overline{\T})$.
This means in particular that  the restriction of $\overline{\sigma_{qs}}$ on $\overline{\T}$ is equivalent to $\sigma_0^{\times n_0}\times\sigma_2^{\times n_2}$ (that is, there are no factors of type $\sigma_1)$.

In the cohomology group $H^1(\Gamma,\overline{\T})$, the $\Gamma$-action on $\T$ is induced by $\sigma_{qs}$. Thus, by Proposition \ref{prop:torus-cohom}, the group $ H^1(\Gamma,\overline{\T})$ is trivial.  This implies that $\chi_U^*$ is injective.
\end{proof}

The map $\Delta_H$ will be a key-ingredient in \S~\ref{sec: non-quasi split case} to determine the existence of a $(G,\sigma)$-equivariant real structure on a horospherical homogeneous space $G/H$ when $\sigma$ is a non quasi-split real group structure on $G$.

\section{Equivariant real structures} \label{sec:equiv real structures}

In this section, we recall the notion of equivariant real structures on $G$-varieties.
We show how to determine if such a structure exists on homogeneous spaces, and, if so, how to use Galois cohomology to determine the set of equivalence classes of these equivariant real structures. We also make some observations on the existence of extensions of these real structures on quasi-homogeneous spaces. These results will be used in the next section.

\begin{definition}
Let $(G,\sigma)$ be a complex algebraic group with a real group structure, and let $X$ be a $G$-variety.
\begin{enumerate}[(i),leftmargin=*]
\item A \emph{$(G,\sigma)$-equivariant real structure} on $X$ is an antiregular involution $\mu$ on $X$ such that
\vspace{-2mm}
\[ \forall g \in G, \; \forall x \in X, \;\; \mu(g \cdot x)=\sigma(g) \cdot \mu(x).\]
\item Two equivariant real structures $\mu$ and $\mu'$ on a $(G,\sigma)$-variety $X$ are \emph{equivalent} if there exists a $G$-equvariant automorphism $\varphi \in \Aut^G(X)$ such that $\mu'=\varphi \circ \mu\circ \varphi^{-1}$.
\end{enumerate}
\end{definition}

\begin{remark}
Let $(X,\mu)$ be a $(G,\sigma)$-equivariant real structure. Whenever the real part $X_0=X(\C)^\mu$ of $X$ is non-empty, it defines a real manifold endowed with an action of the real Lie group $G_0=G(\C)^\sigma$.
\end{remark}

\begin{remark}
If $X$ is a \emph{quasi-homogeneous} $G$-variety, i.e. a $G$-variety with an open orbit $G/H$, then a given equivariant real structure $\mu$ on $G/H$ need not extend to $X$ (as we will see in \S~\ref{sec:extension of real structures}). If $\mu$ extends to $X$, then this extension is unique.
\end{remark}

\begin{lemma} \label{lem: two conditions}
Let $(G,\sigma)$ be a  complex algebraic group with a real group structure, and let  $X=G/H$ be a homogeneous space.
Then $X$ has $(G,\sigma)$-equivariant real structure if and only if there exists $g \in G$ such that these two conditions hold:
\begin{enumerate}
\item \label{eq: sigma compatible}
\emph{$(G,\sigma)$-compatibility condition:}  $\sigma(H)=gHg^{-1}$
\item \label{eq: involution}
\emph{involution condition:}\hspace{15mm} $\sigma(g)g \in H$
\end{enumerate}
in which case one is given by $\mu(kH)=\sigma(k)gH$ for all $k\in G$.
\end{lemma}

\begin{proof}
If $\mu$ defines a $(G,\sigma)$-equivariant real structure on $X$, then $\mu$ is determined by $\mu(eH)=gH$ for some $g \in G$.
Then $\mu$ must be compatible with $\sigma$, i.e.  $\mu(eH)=\mu(hH)=\sigma(h)\mu(eH)$ for all $h \in H$, this yields  condition \ref{eq: sigma compatible}. Also, $\mu$ must be an involution, this yields condition \ref{eq: involution}.
For the converse, if $g \in G$ satisfies the conditions \ref{eq: sigma compatible} and \ref{eq: involution}, then the map $\mu:G\to G$ defined by $\mu(kH)=\sigma(k)gH$  is clearly a $(G,\sigma)$-equivariant real structure on $X$.
\end{proof}

\begin{remark} \label{rk:conjugate}
If $H'$ is conjugate to $H$, then $G/H$ has a $(G,\sigma)$-equivariant real structure $\mu$ if and only if $G/H'$ has a $(G,\sigma)$-equivariant real structure $\mu'$.
Indeed, if $H'=kHk^{-1}$ and $\mu(eH)=gH$, then we can define $\mu'$ by $\mu'(eH')=g'H$ with $g'=\sigma(k)gk^{-1}$.
We  check that  $\sigma(H')=g'H'{g'}^{-1}$ if and only if $\sigma(H)=gH{g}^{-1}$, and that $\sigma(g')g' \in H'$ if and only if $\sigma(g) g \in H$.
\end{remark}

Let us note that, if $N_G(H)=H$, then the condition \ref{eq: involution} of Lemma~\ref{lem: two conditions} holds and we recover \cite[Theorem~2.1]{Akh15}. Also, with the notation of Lemma \ref{lem: two conditions}, we have $X(\C)^\mu \neq \emptyset$ if and only if there exists $k \in G$ such that $k^{-1} \sigma(k)g \in H$.

\begin{example}   \label{ex:Equivariant real structures on toric varieties}
(Equivariant real structures on toric varieties.)\\
For a complete account on toric varieties, we refer the interested reader  to \cite{Ful93}.
Let $(T,\sigma)$ be a complex torus with a real group structure (those were described in Lemma \ref{lem: real form on tori}). Let $X$ be a complex  toric variety with open orbit $X_0 \iso T$.

We start with the case of a homogeneous toric variety, that is, we consider the case $X=X_0$.
It is easy to check, using
Lemma \ref{lem: two conditions} that
the homogeneous space $X_0$ always has a $(T,\sigma)$-equivariant real structure;
we can for instance choose $t_0=1$ and consider $\mu=\sigma$. Now we can find all the equivalence classes of $(T,\sigma)$-equivariant structures on $X_0$. We use the notation of Remark \ref{rem:torus-variety}. Suppose that $T$ is endowed with the real group structure $\sigma=\sigma_0^{\times n_0}\times\sigma_1^{\times n_1}\times\sigma_2^{\times n_2}$. Then each antiregular involution of the form
$\sigma_0^{\times n_0}\times\mu'_1\times\cdots\times\mu'_{n_1}\times\sigma_2^{\times n_2}$, where $\mu'_i=\sigma_1$ or $\tau_1$ for each $i=1,\ldots, n_1$,
defines a $(T,\sigma)$-equivariant real structure on $T$, and no two of these involutions are (equivariantly) equivalent. Moreover, any $(T,\sigma)$-equivariant real structure on $T$ is of this form.

As for the quasi-homogeneous case, by \cite[Theorem 1.25]{Hur11} the equivariant real structure $\mu$ on $X_0$ extends on $X$ if and only if the $\Gamma$-action  on $\X^\vee \otimes_{\Z} \Q$, introduced in Definition \ref{def: Gamma-action on X(T)}, stabilizes the fan of $X$.
\end{example}

\begin{lemma} \label{lem:structure-exist}
Let $(G,\sigma_1)$ be a complex algebraic group with a real group structure, and let  $X=G/H$ be a homogeneous space. Let $\sigma_2=inn(c)\circ \sigma_1$ be an inner twist of  $\sigma_1$ (for some $c \in G$), and suppose that $\sigma_1(H)=H$. Then
\begin{enumerate}[(i),leftmargin=*]
\item $X$ has $(G,\sigma_1)$-equivariant real structure; and
\item $X$ has a $(G,\sigma_2)$-equivariant real structure if and only if there exists $n\in N_G(H)$ such that $c\sigma_1(c)\sigma_1(n)n\in H$.
\end{enumerate}
\end{lemma}

\begin{proof}
The first statement follows from Lemma \ref{lem: two conditions} (take $g=1)$.
For the second statement, note that $\sigma_2(H)=cHc^{-1}$. According to Lemma \ref{lem: two conditions}, the variety $X$ has a $(G,\sigma_2)$-equivariant real structure if and only if there exists $g\in G$ such that $gHg^{-1}=\sigma_2(H)=cHc^{-1}$, and $g\sigma_2(g)\in H$.
But $gHg^{-1}=cHc^{-1}$ if and only if there exists $n\in N_G(H)$ such that $g=cn$.
Then the second condition in Lemma \ref{lem: two conditions} yields $\sigma_2(g)g=c\sigma_1(g)c^{-1}g=c\sigma_1(c)\sigma_1(n)n \in H$. The converse is straightforward, it suffices to take $g=cn$.
\end{proof}

In the next proposition, we use Lemma \ref{lem:structure-exist} to give a cohomological condition to determine the existence of an equivariant real structure on $G/H$. Because of the well-known tables of structures on semisimple algebraic groups, this cohomological interpretation is particularly well-adapted to calculate examples.

\begin{proposition} \label{prop:coho condition}
Let $(G,\sigma_{qs})$ be a complex reductive algebraic group with a quasi-split real group structure, and let $\sigma=\inn_c \circ \sigma_{qs}$ be an inner twist of $\sigma_{qs}$ (for some $c \in G$). Let $X=G/H$ be a homogeneous space, and assume that $N_G(H)/H$ is abelian and $\sigma_{qs}(H)=H$. Then
\begin{enumerate}[(i),leftmargin=*]
\item $X$ has a $(G,\sigma_{qs})$-equivariant real structure; and
\item $X$ has a $(G,\sigma)$-equivariant real structure if and only if $\Delta_H(\sigma)$ is trivial.
\end{enumerate}
\end{proposition}

\begin{proof}
Unraveling the definition of $\Delta_H$ in \S~\ref{sec: useful map}, we verify that the cohomological condition $\Delta_H(\sigma)$ being trivial is equivalent to the second condition in Lemma \ref{lem:structure-exist}.
The result then follows directly from Lemma \ref{lem:structure-exist}.
\end{proof}

\begin{remark}
In \cite{BG2018}, Borovoi and Gagliardi consider the case where $X$ is a quasi-projective $G$-variety over an arbitrary field of characteristic zero which admits a $(G,\sigma_{qs})$-equivariant real structure. They obtain a similar cohomological criterion for the existence of a $(G,\sigma)$-equivariant real structure on $X$, but their point of view differs from ours in the fact that they use exclusively Galois cohomology while we go through group-theoretical considerations (Lemmas \ref{lem: two conditions} and \ref{lem:structure-exist}).
\end{remark}

\begin{remark}
If $(G,\sigma)$ is a simply-connected simple algebraic group with a real group structure, then the Tits class of $(G,\sigma)$ is often trivial (see tables in Appendix \ref{sec: tables}) in which case by Proposition \ref{prop:coho condition} the existence of a $(G,\sigma)$-equivariant real structure on $G/H$ reduces to the study of the conjugacy class of $H$.
\end{remark}

The previous results provide conditions for the existence of an equivariant real structure on a homogeneous space. On the other hand, the classical way to determine the number of equivalence classes for such structures is via Galois cohomology.

\begin{lemma} \label{lem:Galois H1 to param eq real structures} \emph{(Equivariant real structures and Galois cohomology.)}\\
Let $(G,\sigma)$ be a complex algebraic group with a real group structure, and let  $(X,\mu_0)$ be a homogeneous space with a $(G,\sigma)$-equivariant real structure. The Galois group $\Gamma$ acts on $\Aut^G(X)$ by $\mu_0$-conjugacy. Then the map
\[\begin{array}{ccl}
\hspace{-5mm} H^1(\Gamma,\Aut^G(X)) &\to &\{ \text{equivalence classes of $(G,\sigma)$-equivariant real structures on $X$}\}\\
 \varphi &\mapsto &\ \ \ \varphi \circ \mu_0
 \end{array}\]
 is a bijection that sends the identity element to the equivalence class of $\mu_0$.
\end{lemma}

\begin{proof}
Note that any antiregular automorphism of $X$ is of the form $\mu=\varphi\circ\mu_0$, where $\varphi$ is a regular automorphism on $X$. Now  $\mu$ defines a $(G,\sigma)$-equivariant real structure if and only if $\mu\circ\mu=id$ and $\varphi$ is $G$-equivariant.
With the notation of Definition \ref{def:Galois H1}, this is equivalent to requiring that $\varphi\in Z^1(\Gamma, \Aut^G(X))$. Finally, $\mu=\varphi\circ\mu_0$ is equivalent to $\mu'=\varphi'\circ\mu_0$, with $\varphi,\varphi'\in Z^1(\Gamma, \Aut^G(X))$, if and only if there exists $\psi\in  \Aut^G(X)$ such that  $\psi\circ\mu' \circ \psi^{-1}=\mu$, that is, $\psi\circ\varphi \circ (\ga \psi^{-1})=\varphi'$. This is precisely  the equivalence condition defining $H^1(\Gamma, \Aut^G(X))$.
\end{proof}

The interested reader may also consult \cite[\S~I\!I\!I.1]{Ser02} for more general results on the classification of real structures via Galois cohomology.

If $X=G/H$ is a homogeneous space, then we recall that $\Aut^G(G/H)  \iso N_G(H)/H$ (see e.g. \cite[Proposition 1.8]{Tim11}). In particular, if $N_G(H)=H$, then $\Aut^G(X)$ is the trivial group.

\begin{corollary} \label{cor: H=N(H)}
\emph{(see also \cite[Theorem~4.12]{ACF14})} \
Let $(G,\sigma)$ be a complex reductive algebraic group with a real group structure, and let $X=G/H$ be a homogeneous space such that $N_G(H)=H$.
Then $X$ has a $(G,\sigma)$-equivariant real structure $\mu$ if and only if $\sigma(H)=cHc^{-1}$ for some $c \in G$. Moreover, if $\mu$ exists then it is equivalent to $\mu: gH \mapsto \sigma(g)cH$.
\end{corollary}

\begin{proof}
If $N_G(H)=H$, then Condition \ref{eq: sigma compatible} in Lemma \ref{lem: two conditions} implies Condition \ref{eq: involution}, and so $G/H$ has a $(G,\sigma)$-equivariant real structure if and only if $\sigma(H)$ is conjugate to $H$.
As $\Aut^G(X)  \iso N_G(H)/H= \{1\}$ the uniqueness part of the statement follows from Lemma \ref{lem:Galois H1 to param eq real structures}. The last statement is given by Lemma \ref{lem: two conditions}.
\end{proof}

\begin{example} \label{ex:G/P 1}
Let $(G,\sigma)$ be a complex reductive algebraic group with a real group structure, and let $X=G/P$ be a flag variety. Then $X$ has an  equivariant real structure $\mu$ if and only if $\sigma(P)$ is conjugate to $P$. Moreover, if $\mu$ exists then it is equivalent to $\mu: gP \mapsto \sigma(g)cP$, where $c \in G$ satisfies $\sigma(P)=cPc^{-1}$.
\end{example}

\section{Horospherical varieties} \label{sec:equi real groups strcture for horo}
The main part of this section deals with the question of existence of equivariant real structures on horospherical homogeneous spaces.

Let $H$ be a horospherical subgroup of a reductive algebraic group $G$. The group $\T=N_G(H)/H$ is a torus and we apply the results of the previous sections (where $N_G(H)/H$ is assumed to be abelian) to determine whether there exists an equivariant real structure on $G/H$.
We explain also how to determine the number of equivalence classes of equivariant real structures on $G/H$. Then we recall when the equivariant real structures extend to a given horospherical variety.  Finally, we end this section with the study of equivariant real structures on classical examples of horospherical varieties that arise in algebraic geometry.

\subsection{Setting and first definitions} \label{sec:setting}
We fix once and for all a triple $(G,B,T)$, where $G$ is a complex reductive algebraic group, $B \subseteq G$ is a Borel subgroup and $T \subseteq B$ is a maximal torus.  Let $\SS=\SS(G,B,T)$ be the set of simple roots corresponding to the root system associated with the triple $(G,B,T)$. We denote by $\sigma_0$ a split real group structure on $G$ such that $\sigma_0(B)=B$ and $\sigma_0(T)=T$.

\begin{definition}
A subgroup $H$ of $G$ is \emph{horospherical} if it contains a maximal unipotent subgroup of $G$.
A homogeneous space $G/H$ is \emph{horospherical} if $H$ is a horospherical subgroup of $G$.
\end{definition}

\begin{remark}
Let $\sigma$ be a real group structure on $G$. Then $\sigma$ maps a maximal unipotent subgroup of $G$ to a maximal unipotent subgroup of $G$ and so the set of horospherical subgroups of $G$ is preserved by $\sigma$.
\end{remark}

\begin{example}
Tori and flag varieties are examples of horospherical homogeneous spaces.
Let $U$ be a maximal unipotent subgroup of $\SL_2$, then $\SL_2/U$ is a horospherical homogeneous space isomorphic to the affine plane minus the origin $\A^2 \setminus \{0\}$.
\end{example}

\begin{definition}
A \emph{horospherical $G$-variety} is a normal $G$-variety with an open horospherical $G$-orbit.
\end{definition}

\begin{example} \label{ex:SL2}
It follows from the combinatorial description of spherical embeddings (see e.g. \cite[\S~15.1]{Tim11}) that the horospherical $\SL_2$-varieties with open orbit $\SL_2$-isomorphic to $\SL_2/U$ are the following:
$\A^2 \setminus \{0\}$, $\A^2$, $\P^2$, $\P^2 \setminus \{0\}$, $\Bl_0(\A^2)$, and $\Bl_0(\P^2)$.
\end{example}

Other examples of horospherical varieties can be found in \cite{Pas06,Pas18b}.

\bigskip

We now recall the combinatorial description of the horospherical subgroups given in \cite{Pas08}.
For $I \subseteq \SS$, we denote by $P_I$ the \emph{standard parabolic subgroup} generated by $B$ and the unipotent subgroups of $G$ associated with the simple roots $\alpha \in I$ (this gives a 1-to-1 correspondence between the power set of $\SS$ and the set of conjugacy classes of parabolic subgroups of $G$). In particular, $P_{\emptyset}=B$.
Let $I \subseteq \SS$ and let $M$ be a sublattice of $\X$ such that $\left \langle \chi,\alpha^\vee \right \rangle =0$ for all $\alpha \in I$ and all $\chi \in M$, where $ \left \langle \cdot,\cdot \right \rangle: \X \times \X^\vee \to \Z$ is the duality bracket, and $\alpha^\vee \in \X^\vee$ is the coroot associated to the root $\alpha$; this condition simply means that $\chi$ extends to a character of the parabolic subgroup $P_I$.
Then
\begin{equation*}
H_{(I,M)}= \bigcap_{\chi \in M} \Ker(\chi)
\end{equation*}
is a horospherical subgroup of $G$ whose normalizer is $P_I$ (see \cite[\S~2]{Pas08} for details).

\begin{definition}
The horospherical subgroup $H_{(I,M)}$ defined above is called \emph{standard horospherical subgroup}.
If $H$ is a horospherical subgroup, then by \cite[Proposition 2.4]{Pas08} there exists a unique pair $(I,M)$ as above such that $H$ is conjugate  to $H_{(I,M)}$; we write that $(I,M)$ is the (\emph{horospherical})\emph{datum} of $H$.
\end{definition}

\begin{example} \label{ex:flag varieties 1}
The datum of a parabolic subgroup conjugate to $P_I$ is $(I,\{0\})$. The datum of a maximal unipotent subgroup of $G$ is $(\emptyset ,\X)$.
\end{example}

If $H$ is a horospherical subgroup of $G$, then $P:= N_G(H)$ is a parabolic subgroup of $G$ and the quotient group $\T= P/H \iso \Aut^G(G/H)$ is a torus (see \cite[Proposition 2.2 and Remarque 2.2]{Pas08}).

\subsection{Quasi-split case} \label{sec:quasi-split horo}
We denote by $\sigma_{qs}$ a quasi-split real group structure on $G$ preserving $B$ and $T$ as before.

\begin{lemma} \label{lem:two Gamma actions on PI}
Let  $\Gamma$ act on the set $\SS$ as in Definition \ref{def: Gamma-action on X(T)}.
If $I \subseteq \SS$, then  $\sigma_{qs}(P_I)=P_{\ga{I}}$,
where $\ga  I=\{ \ga s \ | \ s \in I\} \subseteq \SS$.

\end{lemma}

\begin{proof}
First note that $\sigma_{qs}(B)=B$ implies that $\sigma(P_I)$ is a parabolic subgroup of the form $P_J$ for some $J \subseteq \SS$ (see \cite[\S~29.4, Th.]{Hum75}).
Since $\Gamma$ acts on $\SS$ by permutation (see Definition \ref{def: Gamma-action on X(T)}), we have $\ga \alpha \in \SS$ for each $\alpha \in \SS$.
Also,  if $\alpha\in\SS$ and $U_{-\alpha}$ is the corresponding unipotent subgroup of $G$, then $\sigma_{qs}(U_{-\alpha})=U_{-\ga\alpha}$ by definition of $U_{-\alpha}$ and of the $\Gamma$-action on $\SS$.
But $P_I=\left \langle B, U_{-\alpha} \text{ with } \alpha \in I \right \rangle $, and so the last sentence implies the result.
\end{proof}

\begin{proposition}  \label{prop: datum of the conjugate}
Let $H$ be a horospherical subgroup of $G$ with datum $(I,M)$.
Then the horospherical datum of $\sigma_{qs}(H)$ is $\ga (I, M):=(\ga I,\ga M)$ with the $\Gamma$-action on $\X$  introduced in Definition \ref{def: Gamma-action on X(T)}.
Also, $\sigma_{qs}(H_{(I,M)})=H_{\ga  (I,M)}$, and so $\sigma_{qs}(H_{(I,M)})$ is conjugate to $H_{(I,M)}$ if and only if $\sigma_{qs}(H_{(I,M)})=H_{(I,M)}$.
\end{proposition}

\begin{proof}
First of all, by Lemma \ref{lem:two Gamma actions on PI}, we have $\sigma_{qs}(P_I)=P_{\ga I}$. Thus, the normalizer of $\sigma_{qs}(H)$ is $\sigma_{qs}(P_I)=P_{\ga I}$.
Also, by definition of the $\Gamma$-action on the coroot lattice $\X^{\vee}$ (see Definition \ref{def: Gamma-action on X(T)}), we have $\left \langle \ga m, \ga (\alpha^\vee) \right \rangle =\left \langle m,\alpha^\vee  \right \rangle=0$ for all $\alpha \in I$ and $m \in M$. Thus the sublattice $M'$ of $\X$ associated with $\sigma_{qs}(H)$ is $\ga M$. This proves the first part of the proposition.

Now we have $\sigma_{qs}(P_I)=P_{\ga I}$ by Lemma \ref{lem:two Gamma actions on PI}, and
\[ \sigma_{qs}(H_{(I,M)})=\bigcap_{m \in M} \sigma_{qs}\left(  \Ker(m)\right)= \bigcap_{m \in  M} \Ker(\ga m)=H_{\ga (I,M)}.\]
Therefore, if $\sigma_{qs}(H_{(I,M)})=H_{\ga (I,M)}$ is conjugate to $H_{(I,M)}$, they have the same horospherical data, that is, $\ga (I,M)=(I,M)$, which finishes the proof.
\end{proof}

\begin{remark}
The previous result can also be obtained as a consequence of \cite[Theorem 3]{CF15} which states that for a spherical subgroup $H \subseteq G$, the subgroups $H$ and $\sigma(H)$ are conjugate in $G$ if and only if the Luna-Vust invariants of $G/H$ are $\Gamma$-stable.
\end{remark}

\begin{remark}
Let $H$ be a spherical subgroup of $G$ such that $\sigma_{qs}(H)$ is conjugate to $H$. If  $N_G(N_G(H))=N_G(H)$ (which holds when $H$ is horospherical), then, by  \cite[Theorem~4.14]{ACF14} (see also \cite[Proposition~2.5]{MJT}),  there  exists a spherical subgroup $H'$ conjugate to $H$ such that $\sigma_{qs}(H')=H'$.
\end{remark}

\begin{corollary} \label{cor:equiv reals tructure if invs preserved}
Let $H$ be a horospherical subgroup of $G$ with datum $(I,M)$.
Then $X=G/H$ has a $(G,\sigma_{qs})$-equivariant real structure if and only if $\ga (I,M)=(I,M)$.
\end{corollary}

\begin{proof}
By remark \ref{rk:conjugate}, we may assume that $H=H_{(I,M)}$.

If  $X=G/H$ has a $(G,\sigma_{qs})$-equivariant real structure, then $H$ and $\sigma_{qs}(H)$ are conjugate by Lemma \ref{lem: two conditions}, and so $\ga (I,M)=(I,M)$ by Proposition \ref{prop: datum of the conjugate}.

Conversely, if $\ga (I,M)=(I,M)$, then $\sigma_{qs}(H)=H$ by Proposition \ref{prop: datum of the conjugate}. Thus we see that the two conditions of Lemma \ref{lem: two conditions} are satisfied for $g=1$, which means that $X$ has a $(G,\sigma_{qs})$-equivariant real structure.
\end{proof}

\begin{remark}
 In the quasi-split case, if there exists a $(G,\sigma_{qs})$-equivariant real structure on $X=G/H$, then there exists one, say $\mu$, such that  $X(\C)^\mu \neq \emptyset$. Indeed, as in the previous corollary,  we may assume that $H=H_{(I,M)}$ and that $\sigma_{qs}(H)=H$.  Then $\mu(eH)=eH$, and so $eH \in X(\C)^\mu$. \end{remark}

\begin{remark} \label{rk:split case}
If $\sigma_{qs}=\sigma_0$ is a split real group structure on $G$, then $X=G/H$ has always a $(G,\sigma_{qs})$-equivariant real structure since $\Gamma$ acts trivially on $\X$.
\end{remark}

\begin{example} \label{ex:H=U}
If $H=U$ is a maximal unipotent subgroup of $G$, then its horospherical datum $(\emptyset, \X)$ is $\Gamma$-stable, and so there always exists a $(G,\sigma_{qs})$-equivariant real structure on $G/U$, namely $\overline{\sigma_{qs}}(gU)=\sigma_{qs}(g)U$.
\end{example}

\begin{example}\label{ex:SL4 part2}
Let $T$ be the maximal torus of $G=\SL_4$ formed by diagonal matrices, and let $B$ be the Borel subgroup formed by upper-triangular matrices.
Let $L_i:T \to \G_m, (t_1,t_2,t_3,t_4) \mapsto t_i$, where $i \in \{1,2,3,4\}$.
Then the simple roots of $(G,B,T)$ are $\alpha_1=L_1-L_2$, $\alpha_2=L_2-L_3$, and $\alpha_3=L_3-L_4$.

Let $P \subseteq G$ be the standard parabolic subgroup
associated with $I=\{\alpha_2\}$, and let $H$ be the kernel of the character $\chi=L_1+L_4$ in $P$.
Then $H$ is a horospherical subgroup of $G$ with datum $(I,M)=(\{\alpha_2\},\Z\left\langle \chi \right\rangle)$.

The group $G$ has two inequivalent quasi-split real group structures, namely $\sigma_s(g)=\overline{g}$ (split) and $\sigma_{qs}(g)=\inn_h \circ \sigma_c(g)$ (non split quasi-split), where $\sigma_c(g)=\leftexp{t}{\overline{g}}^{-1}$ is the compact real group structure on $G$ and \[h=\begin{bmatrix}
0 & 0 &0 &-i \\ 0& 0 & 1 &0\\ 0 & 1&0 &0\\ i & 0 & 0 & 0
\end{bmatrix}.\] 
The $\Gamma$-action on $\X(T)$ induced by $\sigma_s$ is trivial (Remark \ref{rk:split case}), and so $\ga (I,M)=(I,M)$. On the other hand, the $\Gamma$-action induced by $\sigma_{qs}$ is determined by the relations $\ga L_1=-L_4$ and $\ga L_2=-L_3$. Thus, we still have $\ga (I,M)=(I,M)$, but the $\Gamma$-action on $M$ is non-trivial since $\ga \chi=-\chi$. By Corollary \ref{cor:equiv reals tructure if invs preserved}, the homogeneous space $X=G/H$ has a $(G,\sigma_s)$- and a $(G,\sigma_{qs})$-equivariant real structure.
\end{example}

\subsection{General case} \label{sec: non-quasi split case}
We now consider the case where $\sigma$ is any real group structure on the complex reductive algebraic group $G$.

By Theorem \ref{th:ABC Galois theory}, there exists a quasi-split real structure $\sigma_{qs}$ (uniquely defined by $\sigma$, up to equivalence) and an element $c \in G$ such that $\sigma$ is equivalent to $ \inn_c \circ \sigma_{qs}$. We may replace $\sigma$ and $\sigma_{qs}$ in their equivalence classes and assume that $\sigma_{qs}(B)=B$, $\sigma_{qs}(T)=T$, and $\sigma=\inn_c \circ \sigma_{qs}$ to avoid technicalities.
Then we still have a well-defined $\Gamma$-action on $\X$ (preserving $\SS)$ induced by $\sigma_{qs}$ (see Definition \ref{def: Gamma-action on X(T)}).

\begin{lemma}  \label{lem: G1 form implies G0 form}
Let $H$ be a horospherical subgroup with datum $(I,M)$.
If $X=G/H$ has a $(G,\sigma)$-equivariant real structure, then $\ga (I,M)=(I,M)$.
\end{lemma}

\begin{proof}
By Remark \ref{rk:conjugate}, we may assume that $H=H_{(I,M)}$ is a standard horospherical subgroup. If $X$ has a $(G,\sigma)$-equivariant real structure, then Lemma \ref{lem: two conditions}~\ref{eq: sigma compatible} yields that $\sigma(H)$ and $H$ are conjugate. Thus $\sigma_{qs}(H)=\inn_{c^{-1}} \circ \sigma(H)=c^{-1} \sigma(H)c$ and $H$ are also conjugate, and so $\ga (I,M)=(I,M)$ by Proposition \ref{prop: datum of the conjugate}.
\end{proof}

The previous lemma means that the existence of a $(G,\sigma)$-equivariant real structure on $G/H$ implies the existence of a $(G,\sigma_{qs})$-equivariant real structure on $G/H$. However, we will see that these two conditions are not equivalent.

\begin{proposition}  \label{prop:coho cond for horospherical space}
Let $\sigma=\inn_c \circ \sigma_{qs}$ be a real group structure on $G$ as above.
Let $H$ be a horospherical subgroup of $G$ with datum $(I,M)$, and assume that $X=G/H$ has a $(G,\sigma_{qs})$-equivariant real structure.
Then $X$ has a $(G,\sigma)$-equivariant real  structure if and only if $\Delta_H(\sigma)$ is trivial, where $\Delta_H$ is the map defined in \S~\ref{sec: useful map}.
\end{proposition}

\begin{proof}
First, by Remark \ref{rk:conjugate}, we may assume that $H=H_{(I,M)}$ is a standard horospherical subgroup. By assumption, $X$ has a $(G,\sigma_{qs})$-equivariant real structure, and so $\sigma_{qs}(H)=H$ by Proposition \ref{prop: datum of the conjugate}.
Now the result follows from Proposition \ref{prop:coho condition}.
\end{proof}

We now summarize our main results in the following theorem.

\begin{theorem}  \label{th:main results}
Let $H$ be a horospherical subgroup of $G$ with datum $(I,M)$.
Let $\sigma=\inn_c \circ \sigma_{qs}$ be a real group structure on $G$ where $\sigma_{qs}$ is the corresponding quasi-split real group structure preserving $B$ and $T$.
Then the following four conditions are equivalent:
\begin{enumerate}[(i),leftmargin=*]
\item \label{item main th i} $G/H$ has a $(G,\sigma_{qs})$-equivariant real structure;
\item \label{item main th ii} $\ga (I,M)=(I,M)$;
\item \label{item main th iii} $H$ is conjugate to $\sigma(H)$;
\item \label{item main th iv} $H$ is conjugate to $\sigma_{qs}(H)$.
\end{enumerate}
Moreover $G/H$ has a $(G,\sigma)$-equivariant real structure if and only if the (equivalent) conditions (i)-(iv) are satisfied and $\Delta_H(\sigma)$ is trivial, where $\Delta_H$ is the map defined in \S~\ref{sec: useful map}.
\end{theorem}

\begin{proof}
The equivalence of \ref{item main th i} and \ref{item main th ii} is Corollary \ref{cor:equiv reals tructure if invs preserved}. The equivalence of \ref{item main th iii} and \ref{item main th iv} is straightforward since $\sigma=\inn_c \circ \sigma_{qs}$. The equivalence of   \ref{item main th i} and \ref{item main th iv} follows from Proposition \ref{prop: datum of the conjugate} and Lemma \ref{lem: two conditions}.

The second part of the theorem is Proposition \ref{prop:coho cond for horospherical space}.
\end{proof}

\begin{corollary}
With the notation of Theorem \ref{th:main results}, the horospherical homogeneous space $G/H$ has a $(G,\sigma)$-equivariant real structure if $\ga (I,M)=(I,M)$ and one of the following conditions holds:
\begin{enumerate}[(i),leftmargin=*]
\item $H^2(\Gamma,Z(G))=\{1\}$; or
\item $H^2(\Gamma,N_G(H)/H)=\{1\}$.
\end{enumerate}
\end{corollary}

\begin{proof}
Let us note that if (i) or (ii) holds, then the map $\Delta_H$ is trivial.
By Theorem \ref{th:main results}, if $\ga (I,M)=(I,M)$ and $\Delta_H$ is trivial, then  the horospherical homogeneous space $G/H$ has a $(G,\sigma)$-equivariant real structure.
\end{proof}

\begin{example} \label{ex:G/P 2}
(Equivariant real structures on flag varieties.)\\
Let $X=G/P$ be a flag variety, and let $I \subseteq \SS$ be such that the standard parabolic subgroup $P_I$ is conjugate to $P$. By Remark \ref{rk:conjugate} we may assume that $P=P_I$.
We saw in Example \ref{ex:G/P 1} that $X$ has a $(G,\sigma)$-equivariant real structure if and only if $\sigma(P)$ is conjugate to $P$, which we also recover from Theorem \ref{th:main results} since $\Delta_P(\sigma)$ is always trivial.
\end{example}

\begin{example} \label{ex:H=U part 2}
Let $(G,\sigma)$ be a simply-connected simple algebraic group with a real group structure, and let $U$ be  a maximal unipotent subgroup of $G$. By Example \ref{ex:H=U}, there always exists a $(G,\sigma_{qs})$-equivariant real structure on $G/U$. Thus, by Theorem \ref{th:main results} and Lemma \ref{lem:injective-chi}, there exists a $(G,\sigma)$-equivariant real structure on $G/U$ if and only if the Tits class $\delta(\sigma)$ is trivial.
\end{example}

\begin{example}\label{ex:SL4 part3}
We resume Example \ref{ex:SL4 part2}.
The complex algebraic group $G=\SL_4$ has five inequivalent real group structures (see Appendix \ref{sec: tables}): a split one $\sigma_s$ (with real part $\SL_4(\R))$ and an inner twist $\sigma_s'$ (with real part $\SL_2(\mathbb{H}))$, the non split quasi-split one $\sigma_{qs}$ (with real part $\SU(2,2))$ and two inner twists $\sigma_{qs}'$ and $\sigma_c$ (with real parts $\SU(3,1)$ and $\SU(4))$.
By Theorem \ref{th:main results} and Example \ref{ex:SL4 part2}, the homogeneous space $X=G/H$ has a $(G,\sigma)$-equivariant real structure if and only if $\Delta_H(\sigma)$ is trivial.

 Let $\T=P/H \iso \G_m$. If $\sigma=\sigma_s$ resp. $\sigma=\sigma_{qs}$, then the real group structure on $\T$ induced by $\sigma$ is equivalent to $\sigma_0$ resp. to $\sigma_1$ (this follows from a direct computation).
Therefore, by Proposition \ref{prop:torus-cohom}, the group $H^2(\Gamma,\T)$ is trivial when the $\Gamma$-action on $\T$ comes from $\sigma_{qs}$, and so $\Delta_H(\sigma)$ is trivial for the inner twists of $\sigma_{qs}$. It remains to consider the case $\sigma=\sigma_s'$.
Note that for this case, there exists $c\in G$ such that $\sigma_s'=\inn_c\circ \sigma_s$, and that $c\sigma_s(c)\in Z(G)$ and is fixed by $\sigma_s$. This means that $c\sigma_s(c)=\pm 1\in H$. Thus $\Delta_H(\sigma_s')$ is trivial, and so by Theorem \ref{th:main results} there exists a $(G,\sigma_s')$-equivariant real structure on $X$.
\end{example}

\begin{remark}
We will see in the next sections that only $\sigma_{qs}$, the quasi-split inner twist of $\sigma$, matters to count the number of equivalence classes of $(G,\sigma)$-equivariant real structures on $G/H$ or to determine whether  a given  $(G,\sigma)$-equivariant real structure on $G/H$ extends to an equivariant embedding of $G/H$. More precisely, let $X$ be a horospherical variety with open $G$-orbit isomorphic to $X_0=G/H$. Suppose that $\sigma$ is an inner twist of $\sigma_{qs}$. Then if at least one $(G,\sigma)$-equivariant real structure exists on $X_0$, then the number of $(G,\sigma)$-equivariant real structures on  $X_0$ (resp. on $X$) is exactly the same as the number of $(G,\sigma_{qs})$-equivariant real structures on $X_0$ (resp. on $X$).
\end{remark}

\subsection{Number of equivalence classes of equivariant real structures} \label{sec:number of equi real structures}
Let $X=G/H$ be a horospherical homogeneous space, and let $\sigma$ be a real group structure on $G$. As before, we may choose $\sigma$ such that that $\sigma=\inn_{c} \circ \sigma_{qs}$, for some $c \in G$, where $\sigma_{qs}$ is a quasi-split real group structure on $G$ stabilizing $B$ and $T$.

In this section \textbf{we suppose that there exists a $(G,\sigma)$-equivariant real structure $\mu_0$ on $X$.} By \cite[Chp.~I\!I\!I, \S~4.3, Th.~4]{Ser02} the group $H^1(\Gamma,\Aut^G(X))$ is finite. Moreover, since all elements of $H^1(\Gamma,\Aut^G(X))$ are $2$-torsion (see Remark \ref{rk:2 torsion}), we have $H^1(\Gamma,\Aut^G(X)) \iso (\mu_2)^r$ for some non-negative integer $r$ that we want to determine.

By Remark \ref{rk:conjugate}, we may replace  $H$ by a standard horospherical subgroup conjugate to $H$ and assume that $\sigma_{qs}(H)=H$. Then $\sigma_{qs}$ induces a real group structure $\overline{\sigma_{qs}}$ on the torus $\T=N_G(H)/H \iso \Aut^G(X)$. By Lemma \ref{lem: real form on tori}, the real group structure $\overline{\sigma_{qs}}$ is equivalent to
\[\sigma_0^{\times n_0} \times \sigma_1^{\times n_1} \times \sigma_2^{\times n_2} \text{ for some } n_0,n_1,n_2 \in \N \text{ such that } n_0+n_1+2n_2=\dim(\T).\]

\begin{proposition} \label{prop:number of structures}
With the notation and assumptions above, there are exactly $2^{n_1}$ equivalence classes of $(G,\sigma)$-equivariant real structures on $X=G/H$.
\end{proposition}

\begin{proof}
We treat first the case when $\sigma=\sigma_{qs}$.
Note that for this case, there exists a natural equivariant real structure on $X$, namely $\mu_0=\overline{\sigma_{qs}}$ the real structure induced by $\sigma_{qs}$. Then, by Lemma \ref{lem:Galois H1 to param eq real structures},
there is a bijection between the set of equivalence classes of $(G,\sigma_{qs})$-equivariant real structures on $X$ and
$H^1(\Gamma, \Aut^G(X))$, where $\Gamma$ acts on $\Aut^G(X)$ by $\mu_0$-conjugation.
Identifying $\Aut^G(X)$ with $\T=N_G(H)/H$, the $\Gamma$-action on $\Aut^G(X)$ coincides with the $\Gamma$-action on $\T$ induced by $\sigma_{qs}$.
Thus, the number of equivalence classes of $(G,\sigma)$-equivariant real structures on $X$ equals the cardinality of $H^1(\Gamma, \T)$, where the action of $\Gamma$ on $\T$ is given by $\ga(aH)=\sigma_{qs}(a)H$ for any $a\in N_G(H)$. Since $\overline{\sigma_{qs}}$ is equivalent to $\sigma_0^{\times n_0} \times \sigma_1^{\times n_1} \times \sigma_2^{\times n_2}$ by assumption,  the cardinality of $H^1(\Gamma,\T)$ is $2^{n_1}$  by Proposition \ref{prop:torus-cohom}.

We now consider the general case when $\sigma=\inn_c \circ \sigma_{qs}$.
First of all, there is no particular privileged choice of the equivariant  real structure $\mu_0$ on $X$ to define the $\Gamma$-action on $\Aut^G(X)$, as there was in the quasi-split case, namely $\mu_0=\overline{\sigma_{qs}}$. However we assumed that such a $\mu_0$ exists and this will be sufficient for our purposes.
By Lemma \ref{lem:structure-exist}, this means there exists $n\in N_G(H)$ such that $\sigma_{qs}(c)c\sigma_{qs}(n)n\in H$, and then we can take $\mu_0$
defined by  $\mu_0(kH)=\sigma(k)cnH$ for all $k \in N_G(H)$.
Let $\varphi \in \Aut^G(X)$ and $a \in N_G(H)$ such that $\varphi(kH)=kaH$, for all $k\in G$.  Then
\begin{align*}
\hspace{-10mm} &\ga \varphi (kH)=\mu_0 \circ \varphi \circ \mu_0(kH)
                                                       = \mu_0 \circ \varphi(\sigma(k)cnH)\\
                                                       =& \mu_0 (\sigma(k)cnaH)
                                                        = k\sigma(cna)cnH
                                                        = kc \sigma_{qs}(cna)nH.
\end{align*}
Note that $c\sigma_{qs}(cna)nH=c\sigma_{qs}(cn)n(n^{-1}\sigma_{qs}(a)n)H$. Since $c\sigma_{qs}(cn)n\in H$, and $(n^{-1}\sigma_{qs}(a)n)\in N_G(H)$, we find
\[\ga \varphi (kH)=kn^{-1}\sigma_{qs}(a)nH.\]
But $\T$ is abelian, and so $\ga \varphi (kH)=k\sigma_{qs}(a)H$, and  therefore the $\Gamma$-action on $\T$ corresponding to the $\mu_0$-conjugation is given by $\ga (aH)=\sigma_{qs}(a)H$ as in the quasi-split case that we treated first.  Thus the number of equivalence classes of $(G,\sigma)$-equivariant real structures on $X$ is again equal to the cardinality of $H^1(\Gamma,\T)$, where the $\Gamma$-action on $\T$ is the one induced by $\sigma_{qs}$. This finishes the proof.
\end{proof}

\begin{remark} \label{rk:split case number}
Suppose that $\Out(G)=\{1\}$. By Theorem \ref{th:ABC Galois theory} any real group structure on $G$ is an inner twist of the split one $\sigma_s$, and so the induced real group structure on $\T$ is equivalent to $\sigma_0^{\times n_0}$. Thus Proposition \ref{prop:number of structures} implies that when a $(G,\sigma)$-equivariant real structure exists on $G/H$, then it is unique up to equivalence.
\end{remark}

\begin{example} \label{ex:number of real structures on tori}
Let $G=T$ be a torus with a real group structure $\sigma$ equivalent to $\sigma_0^{\times n_0} \times \sigma_1^{\times n_1} \times \sigma_2^{\times n_2}$ for some $n_0,n_1,n_2 \in \N$. Then the number of equivalence classes of $(G,\sigma)$-equivariant real structures on $X=T$ is $2^{n_1}$. This could also be seen directly from Example \ref{ex:Equivariant real structures on toric varieties}.
\end{example}

\begin{corollary} \label{cor:number case G/U}
With the same notation and assumptions as in Proposition \ref{prop:number of structures}, if we moreover assume that $H=U$ is a maximal unipotent subgroup of $G$, then there is a unique equivalence class of $(G,\sigma)$-equivariant real structures on $X=G/U$.
\end{corollary}

\begin{proof}
If $H=U$, then the inclusion $T \hookrightarrow B$ yields an isomorphism $T \iso B/U=\T$, and so the induced real group structure $\overline{\sigma_{qs}}$ on $\T$ coincides with the restriction $(\sigma_{qs})_{|T}$. Since $\sigma_{qs}(B)=B$, this structure must preserve the positive roots of $(G,B,T)$, which means that $(\sigma_{qs})_{|T}$ is equivalent to a product $\sigma_0^{n_0} \times \sigma_2^{n_2}$. Therefore $n_1=0$ and the result follows from Proposition \ref{prop:number of structures}.
\end{proof}

\begin{example}
Let $G=\SL_3$. In Example \ref{ex: forms of SL3} we saw that $G$ has three inequivalent real group structures, namely $\sigma_s$ (split), $\sigma_{qs}$ (non split quasi-split), and $\sigma_c$ (compact, inner twist of $\sigma_{qs})$. Since $\sigma_s$ and $\sigma_{qs}$ stabilize $B$ (up to conjugate), they stabilize also $U$ and so there exists an equivariant real structure on $X=G/U$ in these two cases. Moreover, as the Tits class of $(G,\sigma_c)$ is trivial (see Table 2 in Appendix \ref{sec: tables}), it follows from Theorem \ref{th:main results} and Lemma \ref{lem:injective-chi} that there exists also a $(G,\sigma_c)$-equivariant real structure on $X$. Finally, by Corollary \ref{cor:number case G/U}, each of these three  equivariant real structures is unique up to equivalence.
\end{example}

\begin{example}\label{ex:SL4 part4}
We resume Example \ref{ex:SL4 part3}. We saw that if $\sigma=\sigma_s$ resp. $\sigma=\sigma_{qs}$, then the real group structure on $\T$ induced by $\sigma$ is equivalent to $\sigma_0$ resp. to $\sigma_1$. Thus, by Proposition \ref{prop:number of structures}, there is one equivalence class of $(G,\sigma)$-equivariant real structures on $G/H$ when $\sigma$ is equivalent to an inner twist of $\sigma_s$, and there are two equivalence classes of $(G,\sigma)$-equivariant real structures on $G/H$ when $\sigma$ is equivalent to an inner twist of $\sigma_{qs}$.
\end{example}

\subsection{Extension of the equivariant real structures} \label{sec:extension of real structures}
As before we fix a triple $(G,B,T)$, where $G$ is  a complex reductive algebraic group, $B \subseteq G$ is a Borel subgroup, and $T \subseteq B$ is a maximal torus.
Let $\sigma$ be a real group structure on $G$. In this section we determine when a given $(G,\sigma)$-equivariant real structure on a horospherical homogeneous space $G/H$ extends to a $(G,\sigma)$-equivariant real structure on a horospherical variety whose open orbit is $G/H$.

For the next theorem, we will consider the $\Gamma$-action on the set of colored cones, defined by Huruguen in \cite{Hur11}. Let $\sigma$ be an inner twist of a quasi-split real group structure $\sigma_{qs}$ on $G$, and consider all the colored fans defining $G$-equivariant embeddings $G/H\hookrightarrow X$. Then the $\Gamma$-action on $\X^\vee \otimes_\Z \Q$  defined by $\sigma_{qs}$ (see Definition \ref{def: Gamma-action on X(T)}) induces a $\Gamma$-action on this set of colored fans. In particular, if $\sigma_{qs}=\sigma_0$ is split, then this $\Gamma$-action is trivial.

\begin{theorem} \label{th:Hur-Wed}\emph{(\cite[Theorem 2.23]{Hur11} and \cite[Theorem 9.1]{Wed}.)}\\
Let $\mu$ be a $(G,\sigma)$-equivariant real structure on a \textbf{spherical} homogeneous space $G/H$, and let $X$ be a spherical $G$-variety with open orbit $G/H$. Then the real structure $\mu$ extends on $X$ if and only if the colored fan of the spherical embedding $G/H \hookrightarrow X$ is $\Gamma$-invariant.
\end{theorem}

\begin{remark} \label{rk:alg space}
If the equivariant real structure $\mu$ on $G/H$ extends to $X$, then the corresponding real form $X/\Gamma$ always exists as a real algebraic space but not necessarily as a real variety; see \cite[\S~2.4]{Hur11} for such an example.
\end{remark}

\begin{corollary} \label{cor:extension}
Let $\mu$ be a $(G,\sigma)$-equivariant real structure on a \textbf{horospherical} homogeneous space $G/H$, and let $X$ be a horospherical $G$-variety with open orbit $G/H$. Then the real structure $\mu$ extends on $X$ if and only if the colored fan of the embedding $G/H \hookrightarrow X$ is $\Gamma$-invariant in which case the corresponding real form $X/\Gamma$ is a real variety.
\end{corollary}

\begin{proof}
The condition for the equivariant real structure $\mu$ on $G/H$ to extend to $X$ is given by Theorem~\ref{th:Hur-Wed}. It remains to observe that when $\mu$ extends to $X$, the real algebraic space $X/\Gamma$ is a real variety. This follows from the fact that horospherical $G$-varieties are covered with $\Gamma$-stable quasi-projective horospherical $G$-varieties. A proof of this fact is obtained by using the characterization due to M. Brion of quasi-projectivity for spherical varieties (see e.g. \cite[Corollary 3.2.12]{Per14}) together with the fact that $\Gamma$ is of order $2$.
\end{proof}

\begin{example}
Consider the equivariant embeddings of $SL_2/U$ of Example \ref{ex:SL2}. There are two inequivalent real group structures on $SL_2$: $\sigma_0$ which is split, and $\sigma_1$, whose real part is compact. Note that $\sigma_1$ is an inner twist of $\sigma_0$. We deduce from Example \ref{ex:H=U part 2} and Remark \ref{rk:split case number} (or Corollary \ref{cor:number case G/U}) that there exists a unique equivalence class of $(SL_2,\sigma_0)$-equivariant real structure on $SL_2/U$, but that there is no $(SL_2,\sigma_1)$-equivariant structure on $SL_2/U$ as the Tits class $\delta(\sigma_1)$ is non-trivial.
Thus, any $(SL_2,\sigma_0)$-equivariant real structure on $SL_2/U$ extends to $X$.
\end{example}

\begin{example}\label{ex:SL4 part5}
We resume Example \ref{ex:SL4 part4}. Let $\sigma$ be an inner twist of a non split quasi-split real group structure on $G=\SL_4$. We see from Example \ref{ex:SL4 part2} that the $\Gamma$-action on $N:=M^\vee=\Z \left \langle \chi^\vee \right \rangle $ induced by $\sigma$ satisfies $\ga \chi^\vee=-\chi^\vee$. Let $\mathcal{F} \subseteq N_\Q$ be a colored fan corresponding to a $G$-equivariant embedding $G/H \hookrightarrow Y$.
Then, by Corollary \ref{cor:extension}, a $(G,\sigma)$-equivariant real structure on $G/H$ extends to $Y$ if and only if the colored fan $\mathcal{F}$ is symmetric with respect to the origin of $N_\Q$. It follows from Luna-Vust theory that either $Y=G/H$ (case $\mathcal{F}=\{ (\{0\},\emptyset )\})$ or $Y$ is a $\P^1$-bundle over $G/P$ which is the union of two $G$-orbits of codimension $1$, the two $G$-invariant sections of the structure morphism $Y \to G/P$,  and the open $G$-orbit (case $\mathcal{F}=\{(\Z_+ \left \langle \chi^\vee \right \rangle,\emptyset),(\Z_- \left \langle \chi^\vee \right \rangle,\emptyset)\})$.
\end{example}

\subsection{Smooth projective horospherical varieties of Picard rank 1} \label{sec: smooth proj Picard rank one}
In this section we apply the results obtained in the previous sections to classify the real structures on the smooth projective horospherical $G$-varieties of Picard rank $1$.

Examples of such varieties are given by the flag varieties $X=G/P$, with $P$ a maximal parabolic subgroup, and the odd symplectic Grassmannians; these correspond to the case (3) in Theorem \ref{th:Pasquier classication Picard 1} and were studied for example in \cite{Mih07,Pec13}.
The smooth projective horospherical $G$-varieties of Picard rank $1$ were classified by Pasquier in \cite{Pas09} who proved the following result:

\begin{theorem} \label{th:Pasquier classication Picard 1} \emph{(\cite[Theorem 0.1]{Pas09})}
Let $X$ be a smooth projective horospherical $G$-variety of Picard rank $1$. Then either $X=G/P$ is a flag variety (with $P$ a maximal parabolic subgroup) or $X$ has three $G$-orbits and can be constructed in a uniform way from a triple $(\Dyn(G),\varpi_Y,\varpi_Z)$ belonging to the following list:
\begin{enumerate}
\item $(B_n,\varpi_{n-1},\varpi_n)$ with $n \geq 3$;
\item $(B_3,\varpi_1,\varpi_3)$;
\item $(C_n,\varpi_m,\varpi_{m-1})$ with $n \geq 2$ and $m \in [2,n]$;
\item $(F_4,\varpi_2,\varpi_3)$;
\item $(G_2,\varpi_1,\varpi_2)$,
\end{enumerate}
where $\varpi_Y$, $\varpi_Z$ are fundamental weights of $G$ such that the two closed orbits of $X$ are $G$-isomorphic to the flag varieties $G/P(\varpi_Y)$ and $G/P(\varpi_Z)$. (Here, if $\varpi$ is a fundamental root, $P(\varpi)$ is the parabolic subgroup $P_I$, where $I=\SS\setminus\{\varpi\}$.)
\end{theorem}

We have already looked at equivariant real structures on flag varieties in Examples \ref{ex:G/P 1} and \ref{ex:G/P 2}. Therefore, we will only consider equivariant real structures in the non-homogeneous cases.

\begin{theorem} \label{th:real forms of horo Picard 1}
We keep the notation of Theorem \ref{th:Pasquier classication Picard 1}. Let $\sigma$ be a real group structure on $G$, let $G_0$ be the corresponding real part, and let $X$ be a non-homogeneous smooth projective horospherical $G$-variety of Picard rank $1$ associated with a triple $(\Dyn(G),\varpi_Y,\varpi_Z)$. Then $X$ admits a $(G,\sigma)$-equivariant real structure if and only if $(\Dyn(G),G_0,\varpi_Y,\varpi_Z)$ belongs to the following list:
\begin{enumerate}
\item $(B_n,G_0, \varpi_{n-1},\varpi_n)$ with $G_0=\Spin_{n+4t,n+1-4t}(\R)$ and $n \geq 3$, $t \in \Z$;
\item $(B_3, G_0, \varpi_1,\varpi_3)$ with $G_0=\Spin_7(\R)$ or $\Spin_{3,4}(\R)$;
\item $(C_n, \Sp(2n,\R),\varpi_m,\varpi_{m-1})$ with $n \geq 2$ and $m \in [2,n]$;
\item $(F_4, G_0,\varpi_2,\varpi_3)$ with $G_0$ the real part of one of the three inequivalent real group structures on $F_4$; or
\item $(G_2, G_0,\varpi_1,\varpi_2)$ with $G_0$ the real part of one of the two inequivalent real group structure on $G_2$ (the split one and the compact one).
\end{enumerate}
Moreover, when such a structure exists, then it is unique up to equivalence.
\end{theorem}

\begin{proof}
In cases (1)-(5) of Theorem \ref{th:Pasquier classication Picard 1}, we observe that $\Aut(\Dyn(G))=\{1\}$. Thus, by Remark \ref{rk:split case number}, if a $(G,\sigma)$-equivariant real structure exists on $G/H$, then it is unique up to equivalence. Also, by Theorem \ref{th:ABC Galois theory} \ref{item: dynkin conjugacy classes} any real group structure $\sigma$ on $G$ is an inner twist of the split real group structure $\sigma_0$ on $G$. Therefore the induced $\Gamma$-action on $\X(T)$ is trivial (Remark \ref{rk:trivial split action}). This has two important consequences:
First, by Theorem \ref{th:main results}, the  open orbit $X_0=G/H$ always admits a $(G,\sigma_0)$-equivariant real structure.
Second, by Corollary \ref{cor:extension}, any $(G,\sigma)$-equivariant real structure on $X_0$ extends to $X$.

Moreover, again by Theorem \ref{th:main results}, the homogeneous space $X_0$ admits a $(G,\sigma)$-equivariant real structure if and only if $\Delta_H(\sigma)$ is trivial. (We recall that the Tits class $\delta(\sigma)$ of $(G,\sigma)$ can be found in the tables in Appendix \ref{sec: tables}.) One can check that the cases where $\delta(\sigma)$ is trivial are exactly the cases that appear in the statement of the theorem. Therefore, to finish the proof of this theorem, it suffices to prove that the homomorphism $\chi_H^*$ defined in \S~\ref{sec: useful map} is injective in the cases where $\delta(\sigma)$ is non-trivial. This will indeed imply that $\Delta_H(\sigma)$ is non-trivial.

The cases left to consider are those where $G$ is of type $B_n$, for $n\ge 3$, or of type $C_n$, for $n\ge 2$. Recall that $\T=N_G(H)/H$ is a quotient of the maximal torus $T$ obtained by composing the following homomorphisms:
\[ T \hookrightarrow B \twoheadrightarrow B/U \to N_G(H)/H=\T,\]
where the map $B/U \to N_G(H)/H$ is the map induced by the inclusion $B \hookrightarrow N_G(H)$.
Note that $N_G(H)=BH=TUH=TH$, and so the homomorphism $T \to \T$ is indeed onto, with kernel $T \cap H$. Thus the induced real structure $\overline{\sigma_s}$ on $\T$ (which is a $1$-dimensional torus in cases (1)-(3)) is obtained from ${\sigma_s}_{|T} \sim \sigma_0 \times \sigma_0$, and so $\overline{\sigma_s}$ is equivalent to $\sigma_0$. By Proposition \ref{prop:torus-cohom}, this means that $H^2(\Gamma,\T) \iso \mu_2$. A generator of this group is given by the class of the $\sigma_0$-invariant element $-1$.

We now consider the group $H^2(\Gamma,Z(G))$. In type $B_n$ and $C_n$, note that the center of the simply-connected simple group $G$ is $Z(G) \iso \mu_2$. Since we are considering the cases where the Tits classes are non-trivial, the group $H^2(\Gamma,Z(G))$ is non-trivial, and, since $Z(G)\iso \mu_2$, it is isomorphic to $Z(G)$ on which $\Gamma$ acts trivially.

Recall that the homomorphism $\chi_H^*: H^2(\Gamma,Z(G)) \iso Z(G) \to H^2(\Gamma,\T)$ is induced by the homomorphism $\psi: Z(G) \rightarrow \T$. Since $H^2(\Gamma,\T)$ is generated by the class of $-1$, either $\psi$ is injective, and then $\chi_H^*$ is an isomorphism, or else $\psi$ is trivial, and then $\chi_H^*$ is trivial. We will show that $\psi$ is injective, i.e., that $Z(G)$ is not contained in $H$ and this will finish the proof of this theorem.

We are left to prove that $Z(G)$ is not contained in $H$. For this, use the following construction of $H$ from \cite[\S~1.3]{GPPS}. For each triple $(\Dyn(G),\varpi_1,\varpi_2)$ in Theorem \ref{th:Pasquier classication Picard 1}, consider the projective space $\P(V_1\oplus V_2)$, where $V_1$ and $V_2$ are the irreducible $G$-modules with highest weights $\varpi_1$ and $\varpi_2$ respectively.
Let $v_i$ be a highest weight vector of $V_i$ for  $i=1,2$. Then $H$ is the stabilizer of the line generated by $v_1\oplus v_2\in\P(V_1\oplus V_2)$. With this description, we see that $Z(G)$ is contained in $H$ if and only if $Z(G)$ acts on $V_1 \oplus V_2$ by $\pm1$. However, in each case one can check
that $Z(G)$ acts trivially on $V_i$ and by -1 on $V_j$  with $\{i,j\}=\{1,2\}$. Hence, $Z(G)$ is not contained in $H$ and $\chi_H^*$ is an isomorphism.
\end{proof}

\begin{remark}
Let $(\Dyn(G),\varpi_1,\varpi_2)$ be a triple from Theorem \ref{th:Pasquier classication Picard 1}.
Let $\sigma=\sigma_s$ be a split real group structure on $G$, and let $\mu_s$ be a $(G,\sigma_s)$-equivariant real structure on $X$. Then the geometric construction of $X$ given in \cite[\S~1]{Pas09} (see also \cite[\S~1.3]{GPPS}) yields a construction of the real part of $(X,\mu_s)$ provided that one takes $\R$ as a base field in the construction instead of $\C$.
\end{remark}

\smallskip

\noindent \textbf{Acknowledgments.}
The authors are grateful to Mikhail Borovoi for stimulating discussions and e-mail exchanges about this project; in particular, the cohomological characterization for the existence part in Theorem \ref{th:1} was inspired by \cite{Bor}. We also thank the anonymous referees for their helpful comments.

\newpage

\appendix

\section[The Tits classes of simply-connected simple real algebraic groups]
        {The Tits classes of simply-connected simple real algebraic groups (by Mikhail Borovoi)}
\label{sec: tables}

\newcommand{\upgam}{\hs^\gamma\hmh}
\newcommand{\upsig}{\hs^\sigma\!\kern-0.6pt}
\newcommand{\upsigA}{\hs^{\sigma_{\!*}}\!}

\newcommand{\BRD}{{\rm BRD}}
\newcommand{\id}{{\rm id}}
\newcommand{\ov}{\overline}

\newcommand{\spm}{\{\pm1\}}

In this appendix we give a list of all pairs $(G,\sigma)$,
where $G$ is a  simply-connected, simple, complex algebraic group
with a real group structure $\sigma$ and the corresponding real form $G_0$,
and we give the Tits class in each case.
For the definition of the Tits class, see  Tits \cite[Section 4.2]{Tits71}
(where the Tits class is denoted by $\iota_G$),
or ``The Book of Involutions'' \cite{KMRT98} before Proposition (31.7)
(where the Tits class is denoted by $t_G$),
or  \S\,\ref{sec: useful map} above (where the Tits class is denoted by $\delta(\sigma)$).
In Tables 1 and 2 at the end of the appendix, we list all real forms $G_0$ of each  $G$
(instead of all possible $\sigma$),
and for each $G_0$ we give the corresponding Tits class, which we denote by $t(G_0)$.

Note that in \cite[\S\,31.A]{KMRT98}, the Tits classes
of all classical simple groups over a field of ``good" characteristic  were computed.
However, over $\R$ it is possible to give
a more explicit description of the Tits classes
in terms of the explicit description of the real forms of complex simple groups.

\smallskip

We explain now how we computed $t(G_0)$ in Tables 1 and 2 (using the tables \cite{Tits67} of Tits).

\subsection{Based root datum}
Let $G$ be a (connected) reductive algebraic group over $\C$.
Let $T$ be a maximal torus of $G$, and let $B$
be a Borel subgroup of $G$ containing $T$.
We say that $(T,B)$ is a {\em Borel pair} in $G$.
Let $\X=\X(T)$ denote the character group of $T$, and let
$\X^\vee=\X^\vee(T)$ denote the cocharacter group of $T$.
There is a canonical pairing
\[\X\times\X^\vee\to \Z,\quad (\chi,x)\mapsto \langle\chi, x\rangle\quad\text{for }\chi\in\X,\ x\in \X^\vee.\]
Let $R=R(G,T)\subset \X$ denote the root system of $G$ with respect to $T$,
and let $R^\vee=R^\vee(G,T)\subset\X^\vee$ denote the coroot system.
Let $\SS=\SS(G,T,B)\subset R\subset\X$ denote the set
of simple roots of $G$ with respect to $T$ and $B$,
and let $\SS^\vee=\SS^\vee(G,T,B)\subset R^\vee\subset \X^\vee$
denote the set of simple coroots.
There is a canonical bijection
\[\SS\to \SS^\vee,\quad \alpha\mapsto\alpha^\vee.\]
The set $\SS$ is the set of vertices of the Dynkin diagram $\Dyn(G)=\Dyn(G,T,B)$.
The quadruple $(\X,\X^\vee,\SS,\SS^\vee)$ is called
the {\em based root datum of} $G$ with respect to $T$ and $B$; we denote it by $\BRD(G,T,B)$.
For details see Springer  \cite[Sections 1 and 2]{Spr79}.

If $(T_1,B_1)$ is another Borel pair in $G$, then there exists an element $g\in G$ such that
\[g\cdot T_1\cdot g^{-1}=T,\quad g\cdot B_1\cdot g^{-1}=B.\]
Moreover, if $g'\in G$ is another such element, then $g'=t'g$ for some $t'\in T$.
Using $g$, we can identify $\BRD(G,T,B)$ with $\BRD(G, T_1,B_1)$
by sending $\chi\in\X(T)$ to $\chi_1\in\X(T_1)$
defined by
$$\chi_1(t_1)=\chi(g\hs t_1\hs g^{-1})\in\C^\times \quad\text{for }t_1\in T_1$$
and by sending $x\in\X^\vee(T)$ to $x_1\in \X^\vee(T_1)$ defined by
\[x_1(z)=g^{-1}\cdot x(z)\cdot g\in T_1\quad\text{for }z\in \C^\times.\]
It is easy to see that if we use $g'=t'g$ instead of $g$,
then we obtain the same identification
of $\BRD(G,T,B)$ with $\BRD(G, T_1,B_1)$.
Thus we can canonically identify the based root data
$\BRD(G,T,B)$ for all Borel pairs $(T,B)$.
Therefore, we may write $\BRD(G)$ for $\BRD(G,T,B)$.

\subsection{The $\ast$-action}
Let $\Gamma$ denote the  Galois group $\Gal(\C/\R)=\{1,\gamma\}$,
where $\gamma$ is the complex conjugation.
Let $G_0$ be a real form of $G$.
It defines the {\em $\ast$-action} of $\Gamma$
on the based root datum $\BRD(G)=\BRD(G,T,B)$;
see  \cite[Section 6.2]{BT65}, or \cite[Section 2.3]{Tits66},
or \cite[Section 3.1]{Tits71}, or \cite[Remark 7.1.2]{Con14}.
For the reader's convenience, we recall a construction of the $\ast$-action.
The real form $G_0$ defines an antiregular involution  $\sigma\colon G\to G$
(for a reductive group, an antiregular involution is the same as
an anti-holomorphic involution, see \cite{Cor19}).
Consider $\sigma(T)$ and $\sigma(B)$.
Then $(\sigma(T),\sigma(B))$ is a Borel pair in $G$.
It follows that there exists an element $g_\sigma\in G$ such that
\[g_\sigma\cdot\sigma(T)\cdot g_\sigma^{-1}=T,\quad g_\sigma\cdot\sigma(B)\cdot g_\sigma^{-1}=B.\]
Moreover, if $g'_\sigma\in G$ is another such element,
then $g_\sigma'=t'g_\sigma$ for some $t'\in T$.
We consider the antiregular  automorphism
\[\sigma_*=\inn(g_\sigma)\circ\sigma\colon\, G\to G,\quad
g\mapsto g_\sigma\cdot\sigma(g)\cdot g_\sigma^{-1}
\quad\text{for }g\in G.\]
Then $\sigma_*(T)=T$ and $\sigma_*(B)=B$.
The automorphism $\sigma_*$ induces an automorphism
\[\sigma_\X\colon \X\to\X,\quad\chi\mapsto \sigma_\X(\chi),\quad\text{where }\,
\sigma_\X(\chi)(t)=\upgam(\chi(\sigma_*^{-1}(t)))\ \,\text{for }\chi\in\X,\ t\in T,\]
and where $\upgam(\chi(\sigma_*^{-1}(t)))\coloneqq\ov{\chi(\sigma_*^{-1}(t))}$,
the bar denoting the complex conjugation in $\C$.
It is easy to see that $\sigma_\X$ does not depend
on the choice of $g_\sigma$, that $\sigma_\X^2=\id_\X$,
and that $\sigma_\X$  preserves $\SS\subset \X$.
Similarly, the automorphism $\sigma_*$
induces an automorphism of $\X^\vee$ preserving $\SS^\vee$,
and thus we obtain an automorphism of order dividing 2
of $\BRD(G)=(\X,\X^\vee,\SS,\SS^\vee)$,
which induces an automorphism of order dividing 2
of the Dynkin diagram $\Dyn(G)=\Dyn(G,T,B)$.
For details see  \cite[Proposition 3.1]{BKLR14}.

In \cite{Tits66} and  \cite{Tits71},  another version of this construction is given.
Namely, there exists a maximal torus $T_0\subset G_0$ defined over $\R$, and so,
when constructing $\sigma_*$\hs, we may start from the {\em $\sigma$-invariant} maximal
torus $T=T_0\times_\R \C$ of $G$.
Then $\sigma(T)=T$, hence, $g_\sigma$ is contained
in the normalizer $\mathcal{N}_G(T)$ of $T$ in $G$,
and the inner automorphism $\inn(g_\sigma)$ acts on $\X(T)$ and on  $\X^\vee(T)$
as the element $g_\sigma T$  of the Weyl group $W(G,T)=\mathcal{N}_G(T)/T$.

\subsection{Irreducible representations}
Let $U$ denote the unipotent radical of $B$.
Then we have a decomposition into a semi-direct product $B=U\rtimes T$.
Every character $\chi$ of $T$ can be extended
to a character of $B$ by extending it trivially to $U$.
We may and shall identify the character groups $\X(T)$ and $\X(B)$,
and we shall regard characters of $T$ also as characters of $B$.

A {\em dominant weight of} $G$ (with respect to $T$ and $B$)
is a character $\lambda\in \X$ such that
$\langle\lambda,\alpha^\vee\rangle\ge0$ for all $\alpha\in\SS$.
It is uniquely determined by these nonnegative integers $\langle\lambda,\alpha^\vee\rangle$.

For any dominant weight $\lambda$ of $G$ with respect to  $T$ and $B$,
let $(\rho_\lambda,V_\lambda)$ denote the irreducible complex
representation of $G$ with highest weight $\lambda$.
This means that $V_\lambda$ is a finite dimensional vector space over $\C$,
$$\rho_\lambda\colon G\to \GL(V_\lambda)$$
is an irreducible representation, and there exist a nonzero eigenvector
$v\in V_\lambda$ for $B$  with
\[\rho_\lambda(b)\cdot v=\lambda(b)\cdot v \quad\text{for all }b\in B,\]
where we regard $\lambda$ as a character of $B$.
We say then that $v$ is an eigenvector for $B$ with weight $\lambda$.
For any dominant weight $\lambda$ of $(G,B,T)$, there exists an irreducible,
finite dimensional, complex representation of $G$
with highest weight $\lambda$, and such a representation is unique
up to equivalence (isomorphism of representations);
see, for instance, \cite[Section 15]{MT11}.

The Galois group $\Gamma$ naturally acts on the set of isomorphism classes
of irreducible complex representations of $G_0$.
Namely, let $(\rho, V)$ be a finite dimensional  representation of $G$
in  a complex vector space $V$ of dimension $n$.
We choose a basis $(e_1,\dots,e_n)$ of $V$;
then we may regard $\rho$ as a homomorphism $G\to\GL(n,\C)$.
We consider a new representation
\[\upsig\rho\colon G\to\GL(n,\C),\quad g\mapsto\upgam(\rho(\sigma^{-1}(g)))\quad\text{for }g\in G,\]
where $\upgam(\rho(\sigma^{-1}(g)))\coloneqq\ov{\rho(\sigma^{-1}(g))}$ and
the bar denotes the complex conjugation
applied to the entries of the matrix $\rho(\sigma^{-1}(g))$.
It is easy to see that up to equivalence, the representation $\upsig\rho$
does not depend on the choice of the basis in $V$.
We write $\upgam\hm\rho$ for $\upsig\rho$.

The Galois group $\Gamma$ acts via the $*$-action
on $\Dyn(G)$  and on the set of dominant weights: $\lambda\mapsto\gamma(\lambda)$,
where we write $\gamma(\lambda)$ for $\sigma_\X(\lambda)$.

\begin{lemma}[well known]
\label{l:gamma-lambda}
$ \upgam(\rho_\lambda)\simeq\rho_{\gamma(\lambda)}$.
\end{lemma}

\begin{proof}
Let $(\rho,V)=(\rho_\lambda\hs,V_\lambda)$ be the irreducible complex
representation of $G$ with highest weight $\lambda$,
and let $v\in V$ be  an eigenvector for $B$ with weight $\lambda$.
We consider the representation $\upgam\hm\rho=\upsig \rho$.
Since the antiregular automorphism $\sigma_*$ of $G$
differs from $\sigma$ by an inner automorphism,
we see that $\upsigA\rho\simeq \upsig \rho$.
We consider the vector $\upgam v\coloneqq\bar v\in V$,
the vector obtained from $v$ by complex conjugation of its coordinates
in the chosen basis $(e_1,\dots,e_n)$.
Then for $b\in B$ we have
\[ (\upsigA\rho)(b)\cdot \upgam v=
\upgam (\rho(\sigma_*^{-1}(b))\cdot v)=\upgam(\rho(b')\cdot v)
\quad\text{for }b\in B,\ \text{where }b'=\sigma_*^{-1}(b)\in B. \]
Since $v$ be  an eigenvector for $B$ with weight $\lambda$,
we have $\rho(b')\cdot v=\lambda(b')\cdot v$.
Thus
\[ (\upsigA\rho)(b)\cdot \upgam v=\upgam (\lambda(\sigma_*^{-1}(b))\cdot v)
=\sigma_\X(\lambda)(b)\cdot\upgam v
\quad\text{for }b\in B.\]
This means that $\upgam v$ is an eigenvector for $B$ with respect to $\upsigA\rho$
with weight $\sigma_\X(\lambda)=\gamma(\lambda)$.
Since the representation $\upsigA\rho$ is clearly irreducible, we conclude that
\[\upgam\hm\rho_\lambda\coloneqq\upsig \rho\simeq \upsigA\rho\simeq\rho_{\gamma(\lambda)}\hs,\]
as required.
\end{proof}

\subsection{Tits algebras}
From now on we assume that $G$ is a {\em simply-connected, simple,} complex algebraic group,
and $G_0$ is a real form of $G$.

Let $\lambda\in\X$ be a dominant weight.
We wish to know whether the complex irreducible representation
$\rho_\lambda$ can be defined over $\R$.
If $\gamma(\lambda)\neq\lambda$,
then  $ \upgam(\rho_\lambda)\not\simeq\rho_\lambda$  by Lemma \ref{l:gamma-lambda},
and hence $\rho_\lambda$ cannot be defined over $\R$.
However, even if $\gamma(\lambda)=\lambda$,  it might happen
that the representation $\rho_\lambda$ cannot be defined over $\R$.
The obstruction is the  Tits algebra of $\lambda$, whose definition we now recall.

We assume that $\gamma(\lambda)=\lambda$.
We write $Z=Z(G)$ and $Z_0=Z(G_0)$ for the centers
of the algebraic groups $G$ and $G_0$, respectively.
The dominant weight $\lambda$ induces a homomorphism $\lambda|_Z\colon Z\to\G_{m,\C}$\hs,
where we write $\G_{m,K}$ for the multiplicative group over $K$ for $K=\R$ and $K=\C$.
The homomorphism $\lambda|_Z$ is defined over $\R$
(because $\upgam\hm\rho_\lambda\simeq\rho_\lambda)$.
The obtained homomorphism $Z_0\to \G_{m,\R}$
induces a homomorphism on second cohomology
\[\lambda_*\colon H^2(\R,Z_0)\to H^2(\R,\G_{m,\R})=\{\pm1\}.\]
By definition, the \emph{Tits algebra of $\lambda$} is
$\lambda_*(t(G_0))\in H^2(\R,\G_{m,\R})=\{\pm 1\}$,
where $t(G_0)\in H^2(\R,Z_0)$ is the Tits class of $G_0$.

\begin{theorem}[\hs{\cite[Theorem 7.2]{Tits71}}\hs]
\label{t:Tits}
Assuming that $\gamma(\lambda)=\lambda$,
the irreducible complex representation
$\rho_\lambda$ of $G_0$ can be defined over $\R$
if and only if
\[\lambda_*(t(G_0))=1\in H^2(\R,\G_{m,\R});\]
otherwise,  $\rho_\lambda\oplus\rho_\lambda$
can be defined over $\R$, but $\rho_\lambda$ cannot.
\end{theorem}

Let  $\alpha_1,\dots,\alpha_n$ be the simple roots of $(G,T,B)$
with the  numbering of Bourbaki \cite{Bourbaki}.
The $k$-th fundamental weight  $\lambda_k\in \X$ is defined
by the equalities   $\langle\lambda_k\hs, \alpha_i^\vee\rangle=\delta_{ki}$,
where  $\delta_{ki}$ is Kronecker's symbol.

A dominant weight $\lambda\in\X$ is called {\em minuscule}, if
$\langle\lambda,\alpha^\vee\rangle\in\{0,\pm 1\}$ for all $\alpha\in R$.
It is known that  any nonzero minuscule weight is one of the fundamental weights.
For details, see Bourbaki \cite[VI.\S\hs1, Exercise 24]{Bourbaki}.

\begin{proposition}[Garibaldi \cite{Gar12}, Section 1, Proposition 7]
\label{p:Garibaldi}
The natural map
\[\prod\lambda_*\colon H^2(\R,Z_0)\to\prod H^2(\R,\G_{m,\R})\]
is injective, where the products run over the minuscule weights $\lambda$
such that $\gamma(\lambda)=\lambda$.
\end{proposition}

\subsection{Computation of the Tits classes}

Recall that $G$ is a {\em simply-connected, simple,} complex algebraic group.
We explain how we compute the Tits class of each real form $G_0$ of $G$.

We write $\Lambda_f$ for the set of fundamental weights of $G$ with respect to $(T,B)$.
The real form $G_0$ of $G$ defines a $*$-action of  $\Gamma=\{1,\gamma\}$  on $\Lambda_f$.
We denote by $\Lambda_f^\Gamma$ the corresponding set of fixed points.

We denote by $\Lambda_f^\R$ the set of fundamental weights $\lambda$
such that the corresponding irreducible representation $\rho_\lambda$ of $G$ can be defined over $\R$.
Then $\Lambda_f^\R\subseteq\Lambda_f^\Gamma$.

\begin{lemma}\label{l:R-Gamma}
$t(G_0)=1$ if and only if  $\Lambda_f^\R=\Lambda_f^\Gamma$.
\end{lemma}

\begin{proof}
If $t(G_0)=1$, then by Theorem \ref{t:Tits}  $\Lambda_f^\R=\Lambda_f^\Gamma$.
If  $t(G_0)\ne 1$, then by Proposition \ref{p:Garibaldi}
there exists a minuscule weight $\lambda\in \Lambda_f^\Gamma$
such that $\lambda_*(t(G_0))\neq 1$,
and it follows from Theorem \ref{t:Tits} that  $\lambda\notin \Lambda_f^\R$.
Thus $\Lambda_f^\R\neq\Lambda_f^\Gamma$.
\end{proof}

We compute $t(G_0)\in H^2(\R,Z(G_\qs))=H^2(\R,Z_0)$.
We first consider the case when $H^2(\R,Z_0)\cong\{\pm 1\}$.
This is the case when $G_0$ is of one of the types
$\AA_{2m-1}$, $\BB_n$, $\CC_n$, $\DD_{2m+1}$, $\EE_7$.
In  \cite{Tits67} for each  $G_0$,
the $*$-action of $\Gamma$ on $\Lambda_f$ is described,
which immediately gives $\Lambda_f^\Gamma$,
and the subset $\Lambda_f^\R\subseteq \Lambda_f^\Gamma$  is described as well.
By Lemma \ref{l:R-Gamma} we have $t(G_0)=1$
if and only if  $\Lambda_f^\R=\Lambda_f^\Gamma$.
If $t(G_0)\ne 1$, then $t(G_0)=-1$.

\begin{example}
Let $G_0$ be of type $^2\AA_{2m-1}$, namely, $G_0=\SU(m+s,m-s)$, where $0\le s\le m$.
Then $\gamma(\lambda_k)=\lambda_{2m-k}$, and hence, $\Lambda_f^\Gamma=\{\lambda_m\}$.
According to Tits \cite[p.~28]{Tits67},
$\lambda_m\in \Lambda_f^\R$ if and only if $s$ is even.
We see that $t(G_0)=1$ if and only if $s$ is even.
Thus $t(\SU(m+s,m-s)\hs)=(-1)^s$.
\end{example}

It remains to compute $t(G_0)$ in the case  $^1\DD_{2m}$.
As usual, we denote by $\mu_n$ the  group of roots of unity of order dividing $n$.
Consider the fundamental weights $\lambda_{2m-1}$ and $\lambda_{2m}$
(they correspond to the two half-spin representations $\rho_{2m-1}$ and $\rho_{2m}$).
They induce homomorphisms
\begin{equation}\label{e:lam-lam}
\lambda_{2m-1}\hs,\,\lambda_{2m}\colon\, Z_{0,\C}\to\G_{m,\C}\tag{A.1}
\end{equation}
and an isomorphism
\begin{equation}\label{e:mu2-mu2}
\lambda_{2m-1}\times\lambda_{2m}\colon\, Z_{0,\C}\isoto\mu_2\times\mu_2\hs.\tag{A.2}
\end{equation}
Since $G_0$ is of type  $^1\DD_{2m}$, the $*$-action is trivial, and therefore
the homomorphisms \eqref{e:lam-lam} and the isomorphism \eqref{e:mu2-mu2} are defined over $\R$.
The  homomorphism
\begin{align*}\label{e:lam*-lam*}
(\lambda_{2m-1})_*\times(\lambda_{2m})_*\colon\, H^2(\R, Z_0)\,\isoto\,&H^2(\R,\mu_2)\times H^2(\R,\mu_2)\\
   \isoto\,   &H^2(\R,\G_m)\times H^2(\R,\G_m)\, =\, \{\pm1\}\times  \{\pm1\}\tag{A.3}
\end{align*}
is an isomorphism.
We identify $Z_0$ with $\mu_2\times\mu_2$ via  \eqref{e:mu2-mu2};
then $H^2(\R,Z_0)$ naturally identifies with \,$\{\pm1\}\times  \{\pm1\}$ \,via  \eqref{e:lam*-lam*}.

We consider the case $G_0=\Spin(2m+2s,2m-2s)$, where $0\le s\le m$.
Then according to Tits \cite[p.~39]{Tits67},
the half-spin representations $\rho_{2m-1}$ and $\rho_{2m}$
can be defined over $\R$ if and only if $s$ is even.
By Theorem \ref{t:Tits},
$t(G_0)$ is $(1,1)$ if $s$ is even, and $(-1,-1)$ otherwise.
Thus $t(\Spin(2m+2s,2m-2s)\hs)=(\hs(-1)^s,(-1)^s)$.

We consider the case $G_0=\Spin^*(4m)$.
Then according to Tits \cite[p.~40]{Tits67},
 \emph{exactly one} of the two half-spin representations
$\rho_{2m-1}$ and $\rho_{2m}$ can be defined over $\R$.
We choose the numbering of simple roots such that $\rho_{2m-1}$
can be defined over $\R$ while $\rho_{2m}$ cannot.
By Theorem \ref{t:Tits},
  $t(\Spin^*(4m)\hs)=(1,-1)$.
\bigskip

\noindent{\bf Acknowledgments.}
The author of the appendix is grateful to Skip Garibaldi
for his answer  \cite{Gar18} to the author's MathOverflow question.
Our exposition is based on Garibaldi's answer.

The appendix was substantially revised during the author's stay at the Institut des Hautes \'Etudes Scientifiques, France (IHES).
The author is grateful to IHES for support and excellent working conditions.

\newpage

\subsection{Notation in the tables}

In the following tables,
$G$ is a simply-connected, simple, complex algebraic group;
$Z(G)$ is the center of $G$;
the real algebraic group $G_0$ corresponds to $(G,\sigma)$,
where $\sigma$ is an antiregular involution of $G$;
the real algebraic group $G_\qs$ is a quasi-split inner twist of $G_0$
(corresponding to ($G,\sigma_\qs$)\hs);
$Z(G_\qs)$ is the center of $G_\qs$;
the abelian group $H^2(\R, Z(G_\qs)\hs)$ is
the second Galois cohomology group of $Z(G_\qs)$ over $\R$, that is,
$H^2(\R, Z(G_\qs)\hs)=H^2(\Gamma, Z(G)\hs)$
with $\Gamma=\Gal(\C/\R)$ acting on $Z(G)$ via $\sigma_\qs$\hs;
and $t(G_0)$ is the Tits class of $G_0$\hs.
The left superscript $1$ or $2$ in the formulas like
$^{1}\AA_n$, $^{2}\AA_n$, $^{1}\DD_n$, $^{2}\DD_n$ denotes the order
of the image of the complex conjugation $\gamma$
in the automorphism group of the corresponding Dynkin diagram.
When the group $H^2(\R, Z(G_\qs)\hs)$ is trivial
(which is the case, for instance, if $\Dyn(G)=\EE_6$),
the details on all  terms are not written,
because  the Tits class is clearly trivial in this case.

\bigskip

\begin{table}[h!]\label{Tab-2}
\centering
\caption{Tits classes for the simply-connected exceptional groups}
\scalebox{0.8}{
\begin{tabular}{llllllllll}
\hline
&\\
 $\Dyn(G)$       &$\phantom{G}$         &$Z(G)$           &$\phantom{Book}G_\qs$         &$\phantom{G_\qs}$
 &${\phm H^2(\R,Z(G_\qs)\hs)}$ &$\phantom{BookB} G_0$ &$t(G_0)$  \\
 & & \\
 \hline\\
$\phantom{^1}\EE_6$ & &$\mu_3$ & &  &$\phmm\phantom{\pm}1$ &  &$\phantom{-1}1$
\\
&&\\
$\phantom{^1}\EE_7$ &  &$\mu_2$ &$\phantom{Book}\EE_{7(7)}$ &$ $ &$\hspace{5mm} \spm $ &
$\begin{cases}
{\ \EE_{7(7)},\ \EE_{7(-25)}}\\ \  {\EE_{7(-133)},\   \EE_{7(-5)}}
\end{cases}
$ &
$\begin{array}{l}
\phantom{-}1\\-1
\end{array}$
\\
&&\\
$\phantom{^1}\EE_8$ & &$\pho1$ & &  &$\phmm\phantom{\pm}1$ &  &$\phantom{-1}1$\\
\\
$\phantom{^1}\FF_4$ & &$\pho1$ & &  &$\phmm\phantom{\pm}1$ &  &$\phantom{-1}1$\\
\\
$\phantom{^1}\GG_2$ & &$\pho1$ & &  &$\phmm\phantom{\pm}1$ &  &$\phantom{-1}1$\\
&&&&&&& \\
 \hline
 &&\\
\end{tabular}}
\end{table}

\newpage

\begin{landscape}
\begin{table}[]
\centering
\caption{Tits classes for the simply-connected classical groups}
\end{table}
\label{Tab-1}
\scalebox{0.87}{
\begin{tabular}{lllclcll}
\hline
&\\
 $\Dyn(G)$    &      &$\phantom{Boo}G$      &$\phantom{Bo}Z(G)$      &$\phantom{Bo}G_\qs$    &${\phm H^2(\R,Z(G_\qs)\hs)}$
 &$\phantom{Bookbo}G_0$ &$t(G_0)$  \\
 & & \\
 \hline\\
$\phantom{^1}\AA_{2m}$   &$m\ge 1\phm$    &$\SL(2m+1)$   &$\phm\mu_{2m+1}$  &                &\hspace{4mm} $1$       &
&$\phantom{-1}1$ \\
& \\
$^{1}\AA_{2m-1}$   &$m\ge1$ &$\SL(2m)$   &$\phm\mu_{2m}$   & $\SL(2m,\R)$    &$\phmm\spm$ &
\!\!$\begin{cases}
\SL(2m,\R)\\ \SL(m,\mathbb{H})
\end{cases}
$ &
$\begin{array}{l}
\phantom{-}1\\-1
\end{array}$
\\
&&&\\
$^{2}\AA_{2m-1}$ &$m\ge 2$ &$\SL(2m)$  &$\phm\mu_{2m}$   &$\SU(m,m)$      &$\phmm\spm$   &$\SU(m+s,m-s)$  &$(-1)^s$
\\
 &  &  &  &\\
 $\phantom{^1}\BB_n$ &$n\ge 2$ &$\Spin(2n+1)$  &$\phm\mu_2$ &$\Spin(n,n+1)$  &$\phmm\spm$ &$\Spin(n+2s,n+1-2s)$ &$(-1)^s$\\
 &&\\
 $\phantom{^1}\CC_n$ &$n\ge 3$ &$\Sp(2n,\C)$ &$\phm\mu_2$  &$\Sp(2n,\R)$ &$\phmm\spm$ &
\!\! $\begin{cases}
\Sp(2n,\R)\\ \Sp(s,2n-s)
\end{cases}
$ &
$\begin{array}{l}
\phantom{-}1\\-1
\end{array}$
\\
&\\
%
$^{1}\DD_{2m+1}$   &$m\ge 2$   &$\Spin(4m+2)$       &$\phm\mu_4$  &$\Spin(2m+1,2m+1)$   &$\phmm\spm$
&$\Spin(2m+1+2s,2m+1-2s)$  &$(-1)^s$
\\
&&\\
%
$^{2}\DD_{2m+1}$  &$m\ge 2$  &$\Spin(4m+2)$  &$\phm\mu_4$   & $\Spin(2m+2,2m)$      &$\phmm\spm$ &
\!\!$\begin{cases}
\Spin(2m+2+2s,2m-2s)\\ \Spin^*(4m+2)
\end{cases}
$ &
$\begin{array}{l}
\phantom{-}1\\-1
\end{array}$
\\
&&&\\

$^{1}\DD_{2m}$  &$m\ge 2$  &$\Spin(4m)$  &$\phmm\mu_2\times\mu_2$   & $\Spin(2m,2m)$   & $\phmm\spm\times\spm$   &
\!\!$\begin{cases}
\Spin(2m+2s,2m-2s)\\ \Spin^*(4m)
\end{cases}
$ &
$\begin{array}{l}
\!\!\!\!\!\!\!\!\!\!((-1)^s,(-1)^s)\\ \!\!\!\!{(1,-1)}
\end{array}$
\\
&&&\\

$^{2}\DD_{2m}$  &$m\ge 2$  &$\Spin(4m)$  &$\phmm\mu_2\times\mu_2$   &$\Spin(2m+1,2m-1)$     &\hspace{4mm} $1$
&$\Spin(2m+1+2s,2m-1-2s)$  &$\phantom{-1}1$
\\
&&\\

\\
 \hline
 &&\\
\end{tabular}}
\end{landscape}

\newpage

\bibliographystyle{alpha}

\end{document}